\pdfoutput=1
\documentclass{article}

\usepackage{libertine}
\usepackage[T2A,T1]{fontenc}
\usepackage[utf8]{inputenc}      %
\usepackage[libertine]{newtxmath}
\usepackage[scaled=0.83]{beramono}
\usepackage{eucal}
\usepackage[babel]{microtype}
\usepackage[main=american]{babel}
\usepackage[strict]{csquotes}
\usepackage{alphabeta}              %
\usepackage[fleqn,leqno]{mathtools} %
\usepackage[fleqn,leqno]{amsmath}   %
\usepackage[acronym,toc,nomain]{glossaries}  %
\usepackage{hyperref}               %
\usepackage[capitalise]{cleveref}   %
\usepackage[dvipsnames]{xcolor}     %
\usepackage[ocgcolorlinks]{ocgx2}[2017/03/30]	%

   \let\digamma\relax                   %
\usepackage{amsfonts}
\usepackage{amsthm}
\usepackage{booktabs}
\usepackage{changepage}
\usepackage{diagbox}
\usepackage{dynkin-diagrams}
\usepackage{enumitem}
\usepackage{parskip}
\usepackage{pgfplots}
\pgfplotsset{compat=1.16}     %
\usepackage{rotating}
\usepackage{siunitx}
\usepackage{subcaption}
\usepackage{thmtools}
\usepackage{tikz}
\usepackage{tikz-cd}
\usetikzlibrary{babel}
\usepackage[colorinlistoftodos, textsize = footnotesize]{todonotes}
\usepackage{xparse}
\usepackage{xspace}
\usetikzlibrary{calc,tqft,decorations.markings}

\usepackage[backend=biber,maxbibnames=99,doi=false,sortcites]{biblatex}
\newtoggle{nome}
\toggletrue{nome}
\mathtoolsset{centercolon}

\newcounter{todocounter}
\DeclareDocumentCommand\addreference{g}{\stepcounter{todocounter}\todo[color = blue!30]{\thetodocounter. Add reference\IfNoValueF{#1}{: #1}}\xspace}
\DeclareDocumentCommand\checkthis{g}{\stepcounter{todocounter}\todo[color = red!50]{\thetodocounter. Check this\IfNoValueF{#1}{: #1}}\xspace}
\DeclareDocumentCommand\fixthis{g}{\stepcounter{todocounter}\todo[color = orange!50]{\thetodocounter. Fix this\IfNoValueF{#1}{: #1}}\xspace}
\DeclareDocumentCommand\expand{g}{\stepcounter{todocounter}\todo[color = green!50]{\thetodocounter. Expand\IfNoValueF{#1}{: #1}}\xspace}

\tikzset{every loop/.style={min distance=1cm, looseness=30}}
\tikzset{vertex/.style={circle, draw, black, minimum size=6pt, inner sep=0pt}}
\tikzset{half-vertex/.style={semicircle, draw, fill, black, minimum size=3pt, inner sep=0pt, yshift=1.5pt}}

\tikzset{middlearrow/.style={decoration={markings, mark= at position 0.5 with {\arrow{#1}},}, postaction={decorate}}}

\declaretheoremstyle[
  spaceabove = 3pt,
  spacebelow = 3pt,
  bodyfont = \itshape,
]{first}
\declaretheorem[numberwithin=section, style=first]{theorem}

\declaretheorem[sibling=theorem, style=first]{corollary}
\declaretheorem[sibling=theorem, style=first]{lemma}
\declaretheorem[sibling=theorem, style=first]{proposition}
\declaretheorem[sibling=theorem, style=first]{question}

\declaretheorem[numberwithin=section, style=first, title=Theorem]{alphatheorem}

\declaretheorem[sibling=alphatheorem, style=first, title=Corollary]{alphacorollary}

\declaretheoremstyle[
  spaceabove = 3pt,
  spacebelow = 3pt,
]{second}
\theoremstyle{second}

\declaretheorem[style=second, sibling=theorem]{definition}
\declaretheorem[style=second, sibling=theorem]{example}

\declaretheorem[style=second, sibling=theorem]{remark}

\crefname{conjecture}{Conjecture}{Conjectures}
\crefname{construction}{Construction}{Constructions}

\crefname{alphatheorem}{Theorem}{Theorems}
\crefname{alphaconjecture}{Conjecture}{Conjectures}
\crefname{alphacorollary}{Corollary}{Corollaries}
\crefname{alphaproposition}{Proposition}{Propositions}
\crefname{alphaquestion}{Question}{Questions}

\crefname{figure}{Figure}{Figures}
\crefname{question}{Question}{Questions}

\makeatletter
\def\gitfootnote{\gdef\@thefnmark{}\@footnotetext}
\makeatother

\makeatletter
\DeclareRobustCommand\widecheck[1]{{\mathpalette\@widecheck{#1}}}
\def\@widecheck#1#2{%
    \setbox\z@\hbox{\m@th$#1#2$}%
    \setbox\tw@\hbox{\m@th$#1%
       \widehat{%
          \vrule\@width\z@\@height\ht\z@
          \vrule\@height\z@\@width\wd\z@}$}%
    \dp\tw@-\ht\z@
    \@tempdima\ht\z@ \advance\@tempdima2\ht\tw@ \divide\@tempdima\thr@@
    \setbox\tw@\hbox{%
       \raise\@tempdima\hbox{\scalebox{1}[-1]{\lower\@tempdima\box
\tw@}}}%
    {\ooalign{\box\tw@ \cr \box\z@}}}
\makeatother

\makeindex

\crefformat{enumi}{#2\textup{#1}#3}

\makeatletter
\newcommand{\acro}[3]{%
  \newacronym{#1}{#2}{#3}%
  \@namedef{#1}{{\acrshort{#1}}\xspace}%
  \@namedef{#2}{{\acrlong{#1}}\xspace}%
}

\acro{qft}{QFT}{Quantum Field Theory}
\acro{tqft}{TQFT}{Topological Quantum Field Theory}
\acro{cft}{CFT}{Conformal Field Theory}
\acro{rcft}{RCFT}{Rational Conformal Field Theory}
\acro{sft}{SFT}{Symplectic Field Theory}
\acro{cohft}{CohFT}{Cohomological Field Theory}
\acro{trft}{TropFT}{Tropical Field Theory}

\acro{syz}{SYZ}{Strominger--Yau--Zaslow}
\acro{tuy}{TUY}{Tsuchiya--Ueno--Yamada}
\acro{wdvv}{WDVV}{Witten--Dijkgraaf--Verlinde--Verlinde}
\acro{wz}{WZ} {Wess--Zumino--Witten}

\acro{glsm}{GLSM}{Gauged Linear Sigma Model}
\newcommand   \graph {\gamma}          %
\DeclareMathOperator\SU{SU}                   %
\DeclareMathOperator \Jac    {Jac}              %
\DeclareMathOperator \Pic    {Pic}              %
\DeclareMathOperator \moduli {M}                %

\DeclareMathOperator\NP{NP}
\DeclareMathOperator\Supp{Supp}

\def\Z{\mathbf{Z}}
\def\R{\mathbf{R}}
\def\C{\mathbf{C}}
\DeclareMathOperator\Ker{ker}
\DeclareMathOperator\Coker{coker}

\mathchardef\mhyphen="2D
\newcommand\dash{\nobreakdash-\hspace{0pt}}

\newcommand \dd    {\ensuremath{\mathrm{d}}}      %

\DeclareMathOperator\Aut{Aut}            %

\DeclareMathOperator\HH{H}

\DeclareMathOperator\Spec{Spec}          %

\newcommand \bN {{\ensuremath{\mathbb{N}}}}

\newcommand \bT {{\ensuremath{\mathbb{T}}}}

\newcommand\isomto{\stackrel{\sim}{\smash{\rightarrow}\rule{0pt}{0.4ex}}}

\usepackage{xpatch}

\makeatletter
\patchcmd\blx@bblinput{\blx@blxinit}
                      {\blx@blxinit
                      }{}{\fail}
\makeatother

\usepackage{xpatch}

\makeatletter
\patchcmd\blx@bblinput{\blx@blxinit}
                      {\blx@blxinit
                      }{}{\fail}
\makeatother

\usepackage{xpatch}

\makeatletter
\patchcmd\blx@bblinput{\blx@blxinit}
                      {\blx@blxinit
                      }{}{\fail}
\makeatother

\usepackage{xpatch}

\makeatletter
\patchcmd\blx@bblinput{\blx@blxinit}
                      {\blx@blxinit
                      }{}{\fail}
\makeatother

\usepackage{xpatch}

\makeatletter
\patchcmd\blx@bblinput{\blx@blxinit}
                      {\blx@blxinit
                      }{}{\fail}
\makeatother

\usepackage{xpatch}

\makeatletter
\patchcmd\blx@bblinput{\blx@blxinit}
                      {\blx@blxinit
                      }{}{\fail}
\makeatother

\usepackage{xpatch}

\makeatletter
\patchcmd\blx@bblinput{\blx@blxinit}
                      {\blx@blxinit
                      }{}{\fail}
\makeatother

\usepackage[margin=1.2in]{geometry}

\title{Can one hear the shape of a lattice random walk?}
\author{Pieter Belmans \and Sergey Galkin \and Swarnava Mukhopadhyay}

\begin{document}

\maketitle

\begin{abstract}
We construct distinct high-dimensional mean-zero finite range lattice random walks
having pairwise-equal return probabilities for all step counts.
The same examples provide pairwise-distinct shapes of discretizations of the standard Laplacian
with pairwise-equal density state functions.

The main contribution is a reconstruction theorem:
to a colored trivalent graph one associates a quantum Clebsch--Gordan polytope,
and this association is a full functor,
in particular from the polytope one can uniquely recover the original graph.
These polytopes appear as moment polytopes of toric degenerations of character varieties (moduli spaces of rank-2 bundles on curves),
yielding a combinatorial non-abelian Torelli theorem.
In symplectic geometry, the reconstruction theorem implies that
monotone Lagrangian tori on odd character varieties associated with these degenerations are pairwise non-Hamiltonian isotopic.
These results, and some of the applications,
arose from the study of mirror symmetry for moduli spaces of vector bundles,
and of the related Laurent phenomenon for mutations of graph potentials.
\end{abstract}

\section{Introduction}
The question
\begin{quote}
  \emph{Can one hear the shape of a drum?}
\end{quote}
was popularized by Mark Kac in \cite{MR0201237}
and has its origins in physics.
The problem asks whether a sequence of frequencies (eigenvalues of a Laplacian)
obtained from a vibrating drum
uniquely determines the shape of the drum.
Earlier, Hermann Weyl answered a similar question about frequencies
determining the area of the drum positively.

However, the question of whether one can hear the shape of the drum has been answered negatively in pioneering work by Milnor \cite{MR0162204},
and subsequent works are surveyed in \cite{MR1136137}.
For discrete graphs, the answer is also generally negative; for instance, almost every tree has a cospectral mate \cite{MR0384582}.
Alternatively, one can think of these questions as \emph{reconstruction theorems},
namely given a sequence of measurable quantities of physical significance,
can we reconstruct the original quality or quantity, such as a shape?

In this paper, we examine a \emph{discrete} version of the above questions
for certain discretizations of Laplacian operators in the set-up of random walks,
using Laurent polynomials, graphs, and polytopes.

\paragraph{The many avatars of a Laurent polynomial}
A Laurent polynomial can be used to encode
\begin{itemize}[noitemsep]
  \item a finite range random walk on a~$d$-dimensional lattice~$L = \Z^d$;
  \item a linear difference operator;
  \item a character of a representation;
  \item the generating function for a collection of circles on a torus,
    such as Floer potentials in symplectic topology; or
  \item an element of the group algebra~$k[L]$.
\end{itemize}
More generally,
it encodes a finitely supported function (or measure, or distribution),
denoted
\begin{equation}
  \label{equation:p}
  p\colon L \to k,
\end{equation}
for some value set~$k$, which need not even be a ring.
This distribution~$p$ can equivalently be encoded
either by an element~$W = \sum_{l\in L} p(l) [l]$,
or by a \emph{periodic characteristic function}
\begin{equation}
  \phi\colon \R^d/2\pi L \rightarrow \mathbb{C}:\phi(\theta) = \sum_{l\in L} p_l \exp(\mathrm{i}\langle l,\theta\rangle),
\end{equation}
or by \emph{an exponential moment-generating function}
\begin{equation}
  \label{equation:M}
  M\colon\R^d \rightarrow \R:M(u) = \phi(-iu).
\end{equation}
We will be mainly concerned with the case of a finitely supported distribution~$p$,
in which case the \emph{factorial moment-generating function}
\begin{equation}
  \label{equation:W}
  W\colon({\mathbb{C}^{\times}})^d \rightarrow \mathbb{C}:W(z) = \sum_{l\in L} p_l z^l
\end{equation}
can be evaluated at any~$d$-tuple of non-zero complex numbers~$z = (z_1,\dots,z_d)$.
Here~$({{\mathbb{C}^{\times}}})^d=\operatorname{Spec}\mathbb{C}[L]$.

The value set~$k$ can be binary~$\{0,1\}$,
but it can also be the integers~$\Z$,
probability values~$[0,1] \subset \R$,
or complex numbers~$\C$.
Sometimes a lattice~$L$ comes with extra structure, such as a choice of basis~$e_1,\ldots,e_d$,
or an inner product, but often it has no other than the structure
of a (free abelian) group.

We will consider a class of Laurent polynomials,
introduced by the authors in \cite{gp-tqft},
named \emph{graph potentials}
and establish certain reconstruction results for them.
These polynomials are part of a framework for investigating
mirror symmetry for moduli spaces of vector bundles on Riemann surfaces \cite{gp-sympl}.

\paragraph{Invariants of Laurent polynomials}
Recall that the support~$\Supp(p)$ of a Laurent polynomial~$p$
is the finite set of vectors~$l\in L$ such that~$p(l) \neq 0$,
and the \emph{Newton polytope}~$\NP(p)$ is the convex envelope
of~$\Supp(p)$ in the real affine space~$L_{\R} := L \otimes \R = \R^d$.
It is a higher-dimensional analogue of the degree of a univariate polynomial.
Given a linear injection~$b\colon L \to L'$ that maps~$p$ to~$p'$,
the Newton polytope of~$p$ is mapped to the Newton polytope of~$p'$
under the isomorphism~$L_\R \cong L'_\R$.

Another invariant is the sequence 
of \emph{probabilities of return to the origin},
in algebro-geometric context known as a \emph{period sequence}
\begin{definition}
  The period sequence of a distribution~$p$
  is the sequence of numbers~$\Pi_p = \{a_m\}_{m\in \mathbb{N}}$,
  where
  \begin{equation}
    \label{equation:classperiod}
    a_m = \frac {1}{(2\pi )^d} \int_{[-\pi,\pi]^d} \phi(\theta)^m \dd\theta
  \end{equation}
  is the constant term of the~$m$th power of~$W$.
\end{definition}
In the setting of random walks it corresponds to the probability~$\mathbb{P}(S_m = S_0)$
to be in the same place after~$m$ hops.
Here~$\dd\theta$ is the Lebesgue measure on a~$d$-dimensional torus.
We will also denote it by~$\Pi_W$,
when the emphasis on the input data is \eqref{equation:W}
rather than \eqref{equation:p}.

One may also consider other popular integral transforms, such as
\begin{equation}
  K_W(t) = \int_{u\in\R^d} \exp(-t M(u)) \dd u
\end{equation}
resp.
\begin{equation}
  \zeta_W(s) = \int_{u\in\R^d} M(u)^{-s} \dd u.
\end{equation}
Here~$\dd u$ is the Lebesgue measure on~$\R^d$.
They are typically convergent for~$\Re(t)>0$, resp.~$\Re(s)>1$,
if a neighbourhood of the origin is contained inside the Newton polytope,
which we are going to tacitly assume.
Note that~$K_W$ and~$\zeta_W$ are determined by a measure~$M_* \mu_{\R^d}$ on~$\R$,
that is obtained as the push-forward of the Lebesgue measure on~$\R^d$ with respect to \eqref{equation:M},
and are respectively Laplace and Mellin transforms of the respective density functions.

\paragraph{Mathematical and physical interpretations}
In certain applications the invariants~$\Pi$,~$K_W(t)$,~$\zeta_W(s)$
are either mathematically more accessible or physically more meaningful.
Thus one might have the \emph{hope} that the finite distribution~$p$
(or equivalently, the Laurent polynomial~$W$, or some other incarnation)
could be recovered from one of these generating integral functions,
hence the analogy with ``hearing the shape of a drum''.

For example, in symplectic topology the numbers~$a_m$ might count
holomorphic spheres and are accessible also by algebraic methods,
whereas the numbers~$p_l$ count holomorphic discs whose formulation
so far usually requires analytic tools.
In Bloch--Floquet theory,
the generating function for~$a_m$ is a density-of-states function,
which is to some extent experimentally observable,
whereas the distribution~$p_l$ is more about modelling a discrete Laplacian.

\paragraph{Random walks and mirror symmetry}
The bridge between lattice random walks and mirror symmetry for Fano manifolds
was set up in \cite{galkin-split-notes}.
In the recent works \cite{2405.16979,hu-ke-lu-wip}
probabilistic estimates for random walks, respectively the local central limit theorem
and exponential concentration, are used to obtain results on the mirror symmetry side,
namely on the asymptotics of quantum periods of Fano manifolds.
The present paper goes in the opposite direction:
we use techniques from mirror symmetry,
namely the conifold-point analysis of \cite{1404.7388,galkin-split-notes}
and the properties of graph potentials from \cite{gp-tqft},
as predicted by mirror symmetry,
to obtain results on lattice random walks.

\medskip

In the one-dimensional case one may use Cauchy's residue theorem,
Lagrange's inversion formula, and other tools, such as Galois theory.
Michael Larsen recently has proven \cite{2412.19762}
that the answer to the question ``\emph{Can one hear the shape of a random walk}?'' is affirmative up to two minor ambiguities,
which also holds in higher dimensions.\footnote{We added the word 'lattice' to the title of our work to distinguish it from that of Larsen.}

\paragraph{Setting the stage}
The \emph{first ambiguity} is that the probabilities~$a_m$ do not depend on the choice
of the lattice~$L\subset \R^d$ and stay the same if the lattice~$L$
is replaced by an overlattice~$L'\supset L$.
Recall that 
\begin{definition}
 A random walk is called \emph{irreducible}
  if the support of~$p$ is a generating set for the lattice~$L$ as a monoid,
that is for every lattice point $l$ there exists $n$ such that
the probability to reach $l$ in $n$ steps is non-zero.
\end{definition}
A simple resolution to the first ambiguity is to consider only irreducible lattice random walks.
With any random walk one can associate an irreducible one, by restricting our attention
to the sublattice $L'$ generated by its support, provided zero lies in the interior of the Newton polytope.
The latter condition is automatic for mean-zero random walks.
Geometrically if $\mathrm{T}_L = \Spec \C[L] \simeq (\C^\times)^n$ is the torus,
such that elements $l$ of lattice $L$ correspond to characters $z^l\colon\mathrm{T}_L \to \C^\times$,
then the inclusion $L' \to L$ of an index-$k$ sublattice corresponds to an unramified
covering $\pi\colon\mathrm{T}_L \to\mathrm{T}_{L'}$ and the potential $W\colon\mathrm{T}_L \to \C$
can be obtained as a composition of the covering $T_L \to T_{L'}$
with the potential $W'\colon\mathrm{T}_{L'} \to \C$, where $W'$ is given by $p$
considered as a distribution on a sublattice $L'$. In other words, $W = \pi^* W'$ is a pullback of $W'$
obtained by a monomial \emph{change of coordinates} $\pi$. In terms of coordinates $u$ or $\theta$
the transformation $\pi$ is linear.

For example, the non-irreducible walk
on a two-dimensional lattice given by $W = (x+1/x)(y+1/y)$
is a pullback of the simple random walk $W' = z + 1/z + w + 1/w$
via a monomial map $(x,y) \mapsto (z,w) = (xy, x/y)$,
with the simple random walk being mirror dual to $\mathbb{P}^1 \times \mathbb{P}^1$,
and the pullback being mirror dual to an orbifold toric surface, obtained as a quotient $(\mathbb{P}^1 \times \mathbb{P}^1) / \sigma$
by an involution $\sigma$.

Floer potentials of Fano manifolds are expected to be irreducible
(see Property $\mathcal{O}$ and Conjecture $\mathcal{O}$ 
in \cite{1404.7388,galkin-split-notes,MR3536989,MR4384381})
in contrast to known examples of non-irreducible potentials for some Fano orbifolds in op.~cit.

The \emph{second ambiguity} is similar, but slightly more subtle: the Lebesgue measures
$\dd u$ on~$\R^d$ and~$\dd\theta$ on the~$d$-dimensional torus are translation-invariant,
so the change of coordinates~$u \mapsto u + \tau$ for~$\tau\in\R^d$
transforms the potential~$M(u)$ into the potential~$M_\tau(u) = M(u+\tau)$ with equivalent
integral invariants, but the transformed distribution has coefficients
equal to~$p_\tau(l) = \exp(\langle\tau, l\rangle) p(l)$, and thus is clearly distinct
from~$p$. One has~$M_\tau(0) = M(\tau)$, so unless~$M(\tau)=M(0)$, the new distribution
stops being probabilistic. One can either rescale it and consider~$p_\tau/M(\tau)$,
or study in detail the question when~$M(\tau) = M(0)$, or more generally
$M(\tau_1) = M(\tau_2)$.
In the running example, choosing two points
$(5/2,10)$ and $(528/35,112/165)$
on a circle $x + 1/x + y + 1/y = 13$ in $\mathbb{R}_{>0}^2$
we obtain two random walks with potentials
$\frac{2}{65} x + \frac{5}{26} x^{-1} + \frac{1}{130} y + \frac{10}{13} y^{-1}$
and
$\frac{35}{6864} x + \frac{528}{455} x^{-1} - \frac{165}{1456} y - \frac{112}{2145} y^{-1}$
respectively.

Observe that the equality~$M(\tau_1) = M(\tau_2)$ holds precisely when~$\tau_1$
and~$\tau_2$ lie on the same level set of~$M$,
so all of this is essentially a question about the study of
the level sets of the functions:
\begin{equation}
  M(u)\colon \R^d \to \R,
  \quad
  W(z)\colon (\C^\times)^d \to \C,
  \quad \mbox{and} \quad
  \widetilde{M}(u):= W(\exp(u))\colon \C^d \to \C.
\end{equation}
Note that in Bloch--Floquet theory,
the level sets of the functions~$M(u)$,~$W(z)$ and~$\widetilde{M}(u)$ are respectively real algebraic,
complex algebraic,
and complex analytic varieties, known as \emph{Fermi varieties}.
Concretely, for the Laurent potential~$W$ we set
\begin{equation}
  \label{equation:fermi-bloch-intro}
  F_\lambda(W):=W^{-1}(\lambda)=\{z\in(\C^\times)^d\mid W(z)=\lambda\},
  \qquad
  B(W):=\{(z,\lambda)\in(\C^\times)^d\times\C\mid W(z)=\lambda\},
\end{equation}
the \emph{Fermi variety} at energy~$\lambda$ and the \emph{Bloch variety} of~$W$.
The latter is the graph of~$W$, and the former is its slice
$F_\lambda(W)\times\{\lambda\}=B(W)\cap\bigl((\C^\times)^d\times\{\lambda\}\bigr)$;
in particular~$B(W)$ is the union of all Fermi varieties and determines the whole family~$\{F_\lambda(W)\}_{\lambda\in\C}$. 
	The operator-theoretic incarnation of these varieties, due to Gieseker--Kn\"orrer--Trubowitz, is recalled in \cref{section:fermi}. In the operator literature the Bloch variety~$B(q)$ is generally a determinant / Floquet dispersion hypersurface with respect to a periodic Schr\"odinger operator with respect to a lattice at a spectral parameter $q$. In our Laurent-polynomial model there is no separate field potential (it is constant); the graph potential encodes the hopping/discrete Laplacian symbol.  We also refer to Fillman--Liu--Matos \cite{MR4474842} and Faust--Lopez Garcia \cite{MR4859583} for a discussion on Bloch and Fermi varieties from the perspective of periodic graph operators.

The well-known positive-definiteness of covariance as in \cite{1404.7388,galkin-split-notes}
implies the strict convexity of the function~$M$,
which in turn by Morse theory implies that its non-empty level sets are~$(d-1)$-dimensional spheres,
with the notable exception of the global minimum, attained at some point~$\tau_0$
known as \emph{the conifold point} in mirror symmetry
\cite{1404.7388,galkin-split-notes,MR4384381,2405.16979}.

So one reasonable way to fix the second ambiguity
is to demand~$u=0$ to be the global minimum of the function~$M$,
which is equivalent to the condition~$M'(0) = 0$.
For random walks this condition corresponds to the \emph{mean-zero} property,
and is also often tacitly assumed in the literature.
In the probabilistic language, this is equivalent to~$\mathbb{E}[S_1]=0$.
Without this condition,
the period sequence~$(a_m)_m$ has exponentially decaying asymptotics,
but with this condition the decay is only polynomial.

\paragraph{Resolving ambiguities}
A more radical solution to the two obvious ambiguities is
is to ignore them
by reformulating the hope in a more robust way:

\textbf{Question}\quad
\emph{Given a period sequence~$\{a_n\}_{n\in \mathbb{N}}$
and/or other integral transforms such as~$K_W(t)$ or~$\zeta_W(s)$,
can one recover the Newton polytope~$\NP(p) \subset \R^d$
at least as a convex polytope up to affine linear transformations?}

Before we address this question,
we describe another perspective on the objects involved in this question,
allowing us to further set up notation we will need in our main result.

\paragraph{Discrete Laplacians and density of states}
A linear difference operator was one of the avatars of~$p$ listed at the outset;
under the mean-zero condition discussed above it becomes a discrete analogue of the Laplacian,
whose spectrum the period sequence will turn out to encode.
Consider the difference operator~$\Delta_p$ associated to a distribution~$p$ given by the formula:
\begin{equation}
  \label{eqn:difference}
  (\Delta_p q) (l) = \sum_{l'\in L} p(l') (q(l+l') - q(l)),
\end{equation}
where~$q \in \ell^2(L)$.
Under the mean-zero and isotropic covariance assumptions used in this paper,
$\Delta_p$ is a discretization of the standard Laplacian.

Recall that the Laplacian~$\Delta$ of a compact Riemannian manifold has a discrete spectrum~$\{\lambda_n\}_{n\in \mathbb{N}}$,
and the \emph{density of states}~$\rho$ of a Laplacian of a compact manifold is a spectral distribution given by
\begin{equation}
  \rho(\lambda)=\sum_n \delta(\lambda -\lambda_n).
\end{equation}
Consider the operator~$\mathrm{e}^{-t\Delta }\colon\mathrm{L}^2(M) \rightarrow\mathrm{L}^2(M)$,
and let~$K_t(x,y)$ be the \emph{integral heat kernel}, i.e.,
\begin{equation}
  \mathrm{e}^{-t\Delta}f=\int_M K_t(x,y)f(y)\dd y.
\end{equation}
It is well known that the trace of the operator~$\mathrm{e}^{-t\Delta}$ is the Laplace transform of the density of states, i.e.,
\begin{equation}
  Z(t):=\operatorname{Tr}(\mathrm{e}^{-t\Delta})=\int_{0}^{\infty} \mathrm{e}^{-t\lambda}\rho(\lambda)\dd\lambda.
\end{equation}
In the discrete case of the difference operator~$\Delta_p$ associated to an irreducible, mean-zero random walk,
we can regard the generating function of the period sequence~$\Pi_p$
as the Laplace transform of density of states of the difference operator~$\Delta_p$.

Another variant of the Laplacian, relevant to Bloch--Floquet theory,
is the notion of \emph{periodic Schr\"odinger operators}.
Namely, if~$\Gamma \subset \R^d$ is a lattice and~$H$ is a~$\Gamma$-invariant function in~$\mathrm{L}^2(\R^d/\Gamma)$,
a \emph{periodic Schr\"odinger operator} is an operator of the form~$-\Delta +H$ where~$\Delta$ is the Laplacian on~$\R^d/\Gamma$. In this article we are going to be concerned only with functions $H$ translation-invariant with respect the whole continuous torus (as opposed to a discrete subgroup), i.e. constants.
We refer the reader to \cref{section:fermi} for the precise statements and further details.

We now formulate our main result and discuss the interpretations of our results.

\begin{alphatheorem}\label{theorem:principal}
  Let~$m\geq 3$ be an integer.
  Then there exist~$M=M(m)\geq 2$
  discrete subsets~$S_1,\ldots,S_M$ (where~$\#S_j=8m$)
  of a sphere of squared radius~$3$
  inside a standard Euclidean~$3m$-dimensional lattice~$L = \Z^{3m} \subset \R^{3m}$,
  together with their associated
  Laurent polynomials~$W_j = \sum_{l\in S_j} z^l$ with coefficients in~$\{0,1\}$,
  lattice random walks,
  and periodic Schr\"odinger / difference operators~$\mathrm{D}_{W_j}$ as in \eqref{eqn:difference},
  with the following properties
  satisfied by each single entity
  \begin{enumerate}[label=(\roman*)]
    \item\label{enumerate:convex-independent}
      all sets~$S_j$ are convex independent;
    \item\label{enumerate:principal-zero-mean}
      all random walks are mean-zero, i.e.,~$\sum_{l\in S_j} l = 0$;
    \item\label{enumerate:principal-approximation}
      the difference operators~$D_{W_j}$
      are approximately equal to continuous Laplacian differential operators:
      a Taylor series expansion recovers the Laplacian from the degree-two terms (followed by higher-order terms);
   \end{enumerate}
    and by every pair
   \begin{enumerate}
    \item\label{enumerate:principal-newton-polytope}
      all sets~$S_j$ are pairwise non-affine equivalent,
      in particular the Newton polytopes of~$W_j$ are pairwise non-isomorphic;
    \item\label{enumerate:principal-equal-invariants}
      all periods~$\Pi_{W_j}$, and other integral invariants~$K_{W_j}(t)$,~$\zeta_{W_j}(t)$ are pairwise-equal;
    \item\label{enumerate:principal-fermi}
      for fixed generic~$\lambda\in\mathbf{C}$,
      the Fermi varieties~$F_\lambda(W_j)$ from \eqref{equation:fermi-bloch-intro}
      attached to the different Laurent polynomials~$W_j$ above
      are pairwise birational, but pairwise non-isomorphic as subvarieties of the torus~$(\C^\times)^{3m}$,
      i.e.\ no monomial change of coordinates carries one to another,
      but there is a birational change of coordinates that does.
  \end{enumerate}
\end{alphatheorem}

An informal way to describe the above theorem is the slogan:
{\em One cannot hear the shape of a high-dimensional lattice random walk},
thus answering the question of the title negatively.
We refer the reader to \cref{section:random-walks} for
interpretations of \cref{theorem:principal} in the context of random walks.

Point \cref{enumerate:principal-newton-polytope} is a corollary of \cref{theorem:main},
which we will discuss shortly,
and which is proved in \cref{section:reconstruction}.
Point \cref{enumerate:principal-equal-invariants} is proven in \cref{section:gp}.
Point
\cref{enumerate:principal-fermi}
follow from the discussion in \cref{section:fermi}.

\paragraph{Graph and polytopes}
The main technical ingredient of this article that allows us to prove
\cref{theorem:principal} is the study
of \emph{quantum Clebsch--Gordan polytopes}
$P(\graph) \subset \R^{E(\graph)}$,
a class of convex lattice polytopes
associated with graphs~$\graph$
with edge set~$E$ and vertex set~$V$,
and the proof of \cref{theorem:main},
which gives an algorithmic recovery of graphs from their polytopes.

The polytopes~$P(\graph)$ and their integer points
turn out to be of central interest in mathematics,
and already appeared in various contexts at least since the 1980s.
For example:
\begin{itemize}[noitemsep]
  \item They define the multiplicative or quantum Horn equalities
    in the works of Agnihotri--Woodward and Belkale \cite{MR1671192,MR2373354}.
  \item They are isomorphic to the images of
    the Goldman--Jeffrey--Weitsman integrable systems \cite{MR1204322,MR1279216}
    and the Newton--Okounkov bodies
    arising from conformal blocks in the work of Manon \cite{MR2928457,MR3861806}.
  \item Integer points in their~$k$-dilations are in bijection with
    Bohr--Sommerfeld Lagrangian tori on
    the moduli space of~$\SU(2)$-connections in the symplectic set-up \cite{MR1204322,MR1279216}.
  \item The number~$n_k(\graph)$ of such points is a polynomial in~$k$,
    the so-called Ehrhart polynomial. It coincides with the Hilbert polynomial of
    the moduli space of rank two bundles in the algebro-geometric setup.
    It can be computed by the Verlinde formula \cite{MR1423621},
    in particular it depends only on the genus~$g$ and is equal to the dimension
    of the respective space of conformal blocks of level~$k$
    (or non-abelian theta functions).
    The location of the roots of these polynomials
    was studied by the authors in \cite{MR4399121}.
  \item They appear in works of Kohno and Tsuchiya--Ueno--Yamada \cite{MR1312173,MR1048605} on
    topological invariants from the Knizhnik--Zamolodchikov connection;
    these might be one of the earliest modern references.
\end{itemize}

The following theorem is the main combinatorial reconstruction result,
and it is proven in \cref{section:reconstruction}.

\begin{alphatheorem}
\label{theorem:main}
A graph~$\graph$ can be algorithmically reconstructed
from its quantum Clebsch--Gordan polytope~$P(\graph)$.
In particular,
the polytopes \(P(\graph)\) and \(P(\graph')\) can be mapped to each other
by a pair of (real) affine transformations
  if and only if
the 
graphs \(\graph\) and \(\graph'\) are isomorphic.
\end{alphatheorem}

\paragraph{Polytopes associated to colored graphs}
For geometric applications we need a colored variation on
the construction of quantum Clebsch--Gordan polytopes.
In \cref{section:colorings}, we define colored graphs
and isomorphisms between them, which in particular
allows us to upgrade \cref{theorem:main}
to a description of all isomorphism groups between the respective polytopes:
\begin{alphatheorem}[\cref{theorem:fullness}]
\label{theorem:colored-reconstruction}
Any affine linear isomorphism
between (colored) quantum Clebsch--Gordan polytopes \(P(\graph)\) and \(P(\graph')\)
is induced by a uniquely defined isomorphism of colored graphs.
\end{alphatheorem}
In other words, \cref{theorem:main} can be seen as a conservativity result
for the functor~$P$ discussed in \cref{section:colorings},
whereas \cref{theorem:colored-reconstruction} is a fullness result.
Analogs of \cref{theorem:principal} also hold for colored graphs,
and are discussed in the main body of the text.
We now discuss applications of our results in
discrete random walks, Bloch--Floquet theory, symplectic topology, and in algebraic geometry.

\paragraph{Application: discrete random walks}
In \cite{gp-tqft} we defined a class of Laurent polynomials~$\widetilde{W}_\graph$
using graphs and subsequently discussed their properties.
In \cref{section:random-walks} we attach random walks~$\overline{W}_\graph$ to these Laurent polynomials.
In this language, points \cref{enumerate:principal-newton-polytope} and \cref{enumerate:principal-equal-invariants}
of \cref{theorem:principal} yield the following.
\begin{alphacorollary}[\cref{theorem:cannot-hear}]
Mean-zero lattice random walks~$\overline{W}_{\graph_i}$
associated with a collection of pairwise-distinct graphs~$\graph_i$
have pairwise-distinct shapes, but pairwise-equal probabilities of return to the origin after~$n$~steps for any~$n\in \mathbb{N}$.
\end{alphacorollary}

\paragraph{Application: Schrödinger operators with periodic potentials}
\Cref{section:fermi} deals with an application of the reconstruction theorem
to Bloch--Floquet theory of Schrödinger operators with periodic potentials.
We postpone all the relevant definitions for now and refer the reader to \cref{section:fermi}.
In this language, points \cref{enumerate:principal-newton-polytope} and \cref{enumerate:principal-equal-invariants}
of \cref{theorem:principal} yield the following.

\begin{alphacorollary}
For pairwise-distinct graphs~$\graph_i$, the density-of-states functions of associated periodic Schr\"odinger operators
coincide, but their hoppings (discretizations of Laplacian)
are pairwise-distinct.
\end{alphacorollary}

\paragraph{Application: symplectic topology}
Our initial motivation to prove these reconstruction theorems
was \cite[Corollary~3.10]{gp-sympl} and \cite[Example~4.10]{MR4498821},
where to every graph~$\graph$ of genus~$g$
we associated a monotone Lagrangian torus~$L(\graph)$ in a symplectic manifold~$\mathcal{N}_g$,
and computed Floer potentials of~$L(\graph)$ to be
Laurent polynomials~$W_{\graph}$
with Newton polytopes~$P(\graph)$.
Since Floer potentials are invariants of monotone Lagrangian tori
up to Hamiltonian isotopy, \cref{theorem:fullness}
combined with symplectic results of the cited papers proves the following.
\begin{alphatheorem}{\label{theorem:monotone}}
  Monotone Lagrangian tori~$L(\graph) \subset \mathcal{N}_g$
  are pairwise Hamiltonian non-isotopic
  for distinct graphs~$\graph$ of genus~$g$ provided the graphs have no loops
  or multiple edges. 
\end{alphatheorem}

The symplectic manifold~$\mathcal{N}_g$ of \cref{theorem:monotone}
is known as an odd~$\mathrm{SU}(2)$-character variety.
It parametrizes twisted representations of the fundamental group of a surface~$\Sigma_g$ of genus~$g$,
and its standard/even version parametrizes untwisted representations
and is related to even colorings.
Its algebro-geometric interpretation,
and in particular the association
with the Torelli theorem is discussed next.

\paragraph{Application: tropical non-abelian Torelli theorem}
It is well known that trivalent graphs govern degenerations of a smooth curve
into the deepest strata of the Deligne--Mumford--Knudsen compactification of the moduli space of curves.
Here, we consider polytopes~$P(\graph)$ associated to trivalent graphs which
are polar dual to moment polytopes of toric Fano varieties
that can be realized as degenerations of the moduli space~$\moduli_C(2,\mathcal{L})$
of semistable rank-2 bundles over a Riemann surface~$C$ of genus~$g(C) = g(\graph)$
with fixed determinant~$\mathcal{L}$~\cite{MR2928457} obtained via the graph~$\graph$.
Toric degenerations are often used to tropicalize algebraic varieties.

Analogously to the usual non-abelian Torelli theorem,
that recovers the smooth curve~$C$ from the moduli space~$\moduli_C(2,\mathcal{L})$,
\cref{theorem:colored-reconstruction}
recovers the graph~$\graph$, a tropicalization of a curve~$C$,
from the polytope~$P(\graph)$, a tropicalization of the moduli space~$\moduli_C(2,\mathcal{L})$,
so it can be considered as a
\emph{tropical/combinatorial non-abelian Torelli theorem}:
We refer the reader to~\cref{section:algebraic-geometry}
for precise statements and a more detailed discussion.

\paragraph{Structure of the paper}
In \cref{section:graphs,section:polytopes} we discuss trivalent graphs and their various properties,
followed by a definition of the notion of quantum Clebsch--Gordan (qCG) polytopes associated to these graphs.
\Cref{section:polytopes} also contains key lemmas concerning the rays and vertices of these qCG polytopes
which are used to prove a key reconstruction theorem in \cref{section:reconstruction}.
\Cref{section:gp} introduces graph potentials, whose Newton polytopes are qCG polytopes,
and which play a role in the following three sections.
\Cref{section:changeofcoordinates} addresses part~\cref{enumerate:principal-equal-invariants} of \cref{theorem:principal}.
\Cref{section:random-walks} discusses a reinterpretation of \cref{theorem:principal}
in the setting of discrete random walks.
In \cref{section:fermi}, we recall the notion of Fermi varieties and interpret our result in that context.
\Cref{section:colorings} discusses colored graphs and proves \cref{theorem:colored-reconstruction}.
This has applications in \cref{section:algebraic-geometry}.
The topics in \cref{section:algebraic-geometry} are mostly algebro-geometric in nature,
where we prove a discrete/tropical/combinatorial analogue of the classical abelian and non-abelian Torelli theorems.

No knowledge of algebraic geometry is assumed except for \cref{section:algebraic-geometry}.
\Cref{section:graphs,section:polytopes,section:reconstruction,section:changeofcoordinates}
are essential for the entire paper and require minimal background knowledge.
We recommend reading \cref{section:random-walks} for readers with interests in random walks.
For readers interested in quantum mechanics or solid state physics,
we recommend reading \cref{section:fermi}.
\Cref{section:colorings} contains colored analogues of the results proved in the previous sections
and is essential only for the applications in \cref{section:algebraic-geometry}.
We recommend \cref{section:algebraic-geometry} for readers with interests in algebraic geometry.

\paragraph{Acknowledgements}
We thank Yves Benoist, Indranil Biswas, Promit Ghosal, Yuri Lima, Rodrigo Matos, and Carlos Tomei
for interesting discussions.

This collaboration started
at the Tata Institute for Fundamental Research (TIFR)
during SG's visit in December 2019--March 2020
and PB's visit in February 2020.
We thank TIFR for the very pleasant working conditions.

PB was partially supported by the FWO (Research Foundation---Flanders).
SG was partially supported by CNPq grants PQ 315747 and PQ 308303,
and Coordena\c{c}\~{a}o de Aperfei\c{c}oamento de Pessoal de N\'ivel Superior -- Brasil (CAPES) -- Finance Code 001.
SM was partially supported by
the Department of Atomic Energy, India, under project no. 12-R\&D-TFR-5.01-0500,
and by the ANRF, India (ANRF/ARGM/2025/000957/MTR).
\section{Graphs}
\label{section:graphs}
This section serves to introduce our notation
and intuition for graphs;
we refer the reader to \cite{MR1633290} for details.

A directed multigraph~$\graph$
with vertex set~$V = V(\graph)$
and edge set~$E = E(\graph)$
is given by two maps~$s,t\colon E \to V$, the source and target of an oriented edge.
One can also consider an involution that flips the orientation and exchanges~$s$ with~$t$.

The incidence matrix~$B^\graph$ of size~$|E| \times |V|$
is the difference between linearizations of~$s$ and~$t$:
$B(e,v)$ is~$+1$ if~$v$ is the source of~$e$ ($s(e) = v$),
$-1$ if~$v$ is the target of~$e$ ($t(e) = v$),
and~$0$ otherwise. In case~$e$ is a loop at vertex~$v$,
i.e., an edge going from~$v$ to itself,
one has~$B(e,v)=0$,
but for simplicity in this paper let us consider only graphs without loops.

If~$B$ is considered as an~$|V|$-tuple of vectors in~$\R^E$
then its Gram matrix in the standard Euclidean metric equals~$B^{\mathrm{t}} B = D - A$,
where~$D$ is a diagonal matrix with entries~$D(v,v)$
equal to the valencies of the vertices, and~$A$ is an off-diagonal symmetric
adjacency matrix with~$A(v,v')$ equal to the number of edges incident with both~$v$ and~$v'$.
If there are multiple edges between~$v$ and~$v'$ then~$|A(v,v')|>1$,
but for simplicity in this paper let us consider only graphs without multiple edges.

We say that a graph~$\graph$ is \emph{trivalent} if~$D$ is a scalar matrix with~$3$ on the diagonal,
and will consider only such graphs.
We say that a trivalent graph~$\graph$ has \emph{genus}~$g$ if it is connected
and has~$(2g-2)$ vertices, which also implies that it has~$(3g-3)$ edges.

The definitions of quantum Clebsch--Gordan polytopes from \cref{section:polytopes}
and graph potentials from \cite{gp-tqft}
make sense for general trivalent multigraphs,
with possible leaves, loops, multiple edges,
and the theory of graph potentials crucially uses
these extra degrees of freedom;
however, for the applications in this paper it suffices to consider
only trivalent graphs
\begin{itemize}
  \item that are uniquely determined by their adjacency matrix,
    which has exactly~$3$ non-zero entries (equal to~$1$)
    in each column;
  \item for which every connected component
    has at least~$5$ vertices.
\end{itemize}

For the rest of the paper, our trivalent graphs will have
these two additional properties.
These properties are chosen to simplify the proof and are sufficient for our purposes.

Free abelian groups of~$\Z$-valued functions on~$V$ and~$E$
are known as the groups of~$0$-cochains~$\mathrm{C}^0(\graph) = \Z^V$
and~$1$-cochains~$\mathrm{C}^1(\graph) = \Z^E$, respectively.
They come with standard bases~$\{v^*\}_{v\in V}$,~$\{e^*\}_{e\in E}$
of delta-functions, and with inner products in which these bases are orthonormal.

The random walks of \cref{section:random-walks} can be considered to
happen either in the lattice
$L = \mathrm{C}^1(\graph) = \Z^E$ of~$1$-cochains,
or in some particular sublattice.

The incidence matrix~$B$ is a matrix of a~$\Z$-linear
coboundary map
$\delta\colon\mathrm{C}^0(\graph) \to \mathrm{C}^1(\graph)$,
given by
\begin{equation}
  \delta(v^*) = \sum_{\mathclap{e\in E\colon s(e) = v}}\quad e^* -\quad \sum_{\mathclap{e\in E\colon t(e) = v}}\quad e^*,
\end{equation}
i.e., the image of a vertex is a formal sum of outgoing edges minus the sum of incoming edges.
In our case every column of~$B$ has exactly three non-zero entries equal to~$\pm 1$.

The groups
\begin{equation}
  \mathrm{Z}^V := \mathrm{C}^0(\graph),
  \ \mathrm{Z}^E := \mathrm{C}^1(\graph),
  \ \Ker\left( \delta\colon\mathrm{C}^0(\graph) \to\mathrm{C}^1(\graph) \right)\ \mbox{and}
  \ \Coker\left( \delta\colon\mathrm{C}^0(\graph) \to \mathrm{C}^1(\graph) \right)
\end{equation}
are the groups of~$0$-cochains,~$1$-cochains,
zeroth cohomology~$\mathrm{H}^0(\graph)$
and
first cohomology~$\mathrm{H}^1(\graph)$ of the graph, respectively.

The topological interpretation of the transpose matrix~$B^{\mathrm{t}}$ is of a dual map
$\partial\colon\mathrm{C}_1(\graph) \to \mathrm{C}_0(\graph)$ from~$1$-chains to~$0$-chains,
and its kernel and cokernel are known as first homology~$\mathrm{H}_1(\graph)$
and zeroth homology~$\mathrm{H}_0(\graph)$, respectively. For these matters
one can replace the integers~$\Z$ by any other abelian group, e.g.,~$\R$ or
$\Z/2 \simeq \{\pm 1\}$.
We will later consider polytopes
in~$\mathrm{C}^1(\graph,\R) = \R^E$ and their polar duals in~$\mathrm{C}_1(\graph,\R)$,
and in \cref{section:colorings} we consider coloring functions on vertices,
that are essentially~$0$-chains valued in a group of two elements.

In \cref{section:algebraic-geometry}, the trivalent graphs~$\graph$
are discrete/combinatorial counterparts (and tropical limits)
of Riemann surfaces.
The notion of genus for a graph
coincides with that of genus for a Riemann surface
under this dictionary.

The differential~$\delta\colon\mathrm{C}^0(\graph) \to \mathrm{C}^1(\graph)$
is a counterpart of de Rham differential~$\dd_{\mathrm{dR}}\colon \Omega^0(C) \to \Omega^1(C)$
from functions to differentials on a Riemann surface~$C$.
So analytically
the transpose matrix~$B^{\mathrm{t}}$
can be interpreted as a counterpart of the Hodge-theoretic adjoint map
$\dd_{\mathrm{dR}}^*\colon \Omega^1(C) \to \Omega^0(C)$;
a counterpart of the Hodge metric on~$C$
being two inner products on~$\mathrm{C}^0(\graph) = \Z^V$ and~$\mathrm{C}^1(\graph) = \Z^E$
with respect to which the standard bases~$v^*$ and~$e^*$ are orthonormal.
Then~$B^{\mathrm{t}}$ is interpreted as~$\delta^*\colon\mathrm{C}^1(\graph) \to\mathrm{C}^0(\graph)$,
and the square matrix~$B^{\mathrm{t}} B$, i.e.,~$\delta^* \delta\colon \mathrm{C}^0(\graph) \to \mathrm{C}^0(\graph)$
is a counterpart of the Laplace operator acting on functions on a Riemann surface.

\section{Quantum Clebsch--Gordan polytopes}
\label{section:polytopes}

We start with a definition and a
study of some properties of quantum Clebsch--Gordan polytopes (qCG-polytopes for short) attached to graphs,
that in the end will help to characterize them and reconstruct the respective graphs.

These polytopes are embedded in
the affine space~$\R^{E}$
with coordinates naturally indexed by the edges.
The faces of these polytopes~$P(\graph)$
are parameterized by the vertices~$v$.
For a vertex~$v\in V$
with edges~$i,j,k$ incident to it,
we consider the \emph{quantum Clebsch--Gordan equations}
for the real parameters~$x,y,z$ indexed by~$i,j,k$
which read
\begin{equation}
  \left\{
    \begin{array}{ll}
      x\leq y+z; y\leq x+z\  \mbox{and} \ z\leq x+y, &\text{triangle inequality} \\
      x+y+z\leq 1, & \text{perimeter inequality} \\
      \pm x\pm y\pm z \in \Z.
    \end{array}
  \right.
\end{equation}
These inequalities can be combined as follows.
To a graph \(\graph\) with its set \(V(\graph)\) of vertices
associate the vector space \(S(\graph):=\R^{E}\)
generated by a basis \(e(i)\)
indexed by the edges of \(\graph\).
The \emph{quantum Clebsch--Gordan polytope} \(P(\graph)\) in~$S(\graph)$
is the convex envelope of the vectors
\begin{equation}
  \label{equation:rays}
  p(v,s):=s(i) e(i) + s(j) e(j) + s(k) e(k)\in S(\graph), %
\end{equation}
for all~$v\in V(\graph)$
and all sign choices~\(s(i),s(j),s(k)\in\{\pm1\}\)
\begin{equation}
  \label{equation:sign-choice-no-coloring}
  s(i)s(j)s(k)=1,
\end{equation}
where the edges \(i,j,k\) are incident to the vertex \(v\).

\begin{remark}
The polar dual of $P(\graph)$ is given by $4|V|$ equations,
namely by four quantum Clebsch--Gordan equations for every vertex.
\end{remark}

To avoid confusion between the vertices of the graph and vertices of the polytope,
we will refer to the vertices of the polytope as \emph{rays} in what follows,
and we are going to consider edges of the graph, but not of the polytope.

For a fixed graph~$\graph$
we consider a collection of vectors~$p(v,s)$
associated to all vertices~$v$ and admissible~$s$.
Each of them in the basis \(\{e(i)\}\)
is a permutation of the entries of the vector
\((s(1)1,s(2)1,s(3)1,0,0,...,0)\) where~$s(1)s(2)s(3)=1$.
Consider the Euclidean norm in which the basis \(e(i)\) is orthonormal.

\begin{proposition}
  \label{proposition:scalar-product}
  The scalar product~$(p(v,s),p(v',s'))$ takes four values:
  \begin{itemize}
    \item~$+3$ if and only if~$v=v'$ and~$s=s'$;
    \item~$+1$ if and only if~$v$ and~$v'$ are adjacent along an edge~$e$ and~$s(e) = -s'(e)$;
    \item~$-1$ if and only if either of two possibilities hold
      \begin{itemize}
      \item~$v=v'$ and~$s\neq s'$,
      \item~$v$ and~$v'$ are adjacent along an edge~$e$ and~$s(e) = s'(e)$;
      \end{itemize}
    \item~$0$ otherwise (i.e., if~$v$ and~$v'$ are neither coincident nor adjacent).
  \end{itemize}
\end{proposition}

\begin{proof}
  To compute the scalar product, note that it equals
  \begin{equation}
    (p(v,s),p(v',s')) = \sum_e s(e) s'(e),
    \end{equation}
  where the sum is over the edges adjacent to both~$v$ and~$v'$.
  If the vertices~$v,v'$ are neither coincident nor adjacent, the sum is empty.
  For coincident vertices there are three terms,
  all equal to~$+1$ for~$s'=s$ and one~$+1$, two~$-1$s for~$s'\neq s$.
  For adjacent vertices (in the graphs that we consider) there is a single
  adjacent edge, so a single adjacent term.
\end{proof}

\begin{corollary}
  The vectors~$p(v,s)$ are
  \begin{enumerate}
    \item pairwise-distinct,
      i.e.,~$p(v,s) = p(v',s')$ if and only if~$v=v'$ and~$s=s'$;
    \item equidistant from \(0\);
    \item and convex independent,
      i.e., none of them lies in a convex envelope of others.
  \end{enumerate}
\end{corollary}

\begin{proposition}
  \label{lemma:center}
  The center of mass of the vectors~$p(v,s)$ is the origin, i.e.,
  \begin{equation}
     \sum_{v\in V}\sum_s p(v,s) = 0.
  \end{equation}
\end{proposition}

\begin{proof}
  The set of vectors~$p(v,s)$
  is partitioned into quadruples \(\{p(v,s)\}\) enumerated by the vertices \(v \in V(\graph)\).

  If \(x,y,z,w\) are four distinct admissible sign assignments
  then there is a linear relation
  \begin{equation}
    \label{equation:linear-relation}
    p(v,x) + p(v,y) + p(v,z) + p(v,w) = 0,
  \end{equation}
  so the center of mass of the vectors in any such quadruple is \(0\),
  hence the center of mass of all the vectors is~\(0\).
\end{proof}

We obtain the following corollary to \cref{lemma:center,proposition:scalar-product}.

\begin{corollary}
  The association~$(v,s) \mapsto p(v,s)$
  is a bijection from the set of pairs~$(v,s)$ of a vertex and sign assignment
  to the set of rays of the polytope~$P(\graph)$.
  In particular, the polytope~$P(\graph)$ has~$4 |V(\graph)|$ rays.
\end{corollary}

Let~$N$ denote the lattice spanned by the vectors~$p(v,s)$.
It is a finite-index sublattice in the lattice generated by~$e(i)$,
which we will denote by~$\widetilde{N}$.

The discussion above shows that the lattice~$N$ can be reconstructed
from the polytope~$P(\graph)$; however, to reconstruct the lattice
$\widetilde{N}$ and the vectors~$\pm e(i)$ we need more information,
as neither is immediate to reconstruct.

Instead let us study small linear relations between the vectors,
since this data will be directly computable from the polytope.
Denote the \emph{weight} of a combination~\( w(v,s) \)
of integer coefficients~$w(v,s)$
to be the sum of absolute values~$\sum |w(v,s)|$.
Call a combination \emph{positive}, if all weights are non-negative.
Such relations form a cone in the space of combinations,
which is a free lattice~$\Z^{4|V|}$.
The kernel of the linear evaluation map
\begin{equation}
  \{w(v,s)\} \mapsto \sum w(v,s) p(v,s) \in N
\end{equation}
is the space of linear \emph{relations},
so linear relations are linear combinations
such that the respective vector in~$N$ is zero.

Since the sum of coordinates of every vector~$p(v,s)$ is equal to~$3$ modulo~$4$
it is clear that~$\sum w(v,s)$ is divisible by~$4$ for any relation,
so the smallest weight of a non-zero relation is~$4$.
In this case either~$\sum w(v,s) = 0$ and the relation can be rewritten
either in the form
\begin{equation}
  p(v_1,s_1) + p(v_2,s_2) = p(v_3,s_3) + p(v_4,s_4)
\end{equation}
or in the form
\begin{equation}
  p(v_1,s_1) - p(v_3,s_3) = p(v_4,s_4) - p(v_2,s_2),
\end{equation}
or~$\sum w(v,s) = \pm 4$ and the relation reads
\begin{equation}
  p(v_1,s_1) + p(v_2,s_2) + p(v_3,s_3) + p(v_4,s_4) = 0.
\end{equation}
So both cases can be rephrased as an equality of sums of pairs of vectors,
and additionally in the first case as an equality of differences of pairs of vectors.

Let us compute such sums and differences,
divided by two for convenience. It is straightforward to check the next result.

\begin{lemma}
For all vertices~$v$ incident to edges~$i,j,k$
and all admissible signs~$s\neq s'$ a semi-sum~$(p(v,s) + p(v,s'))/2$
takes six values~$\pm e(i)$,~$\pm e(j)$,~$\pm e(k)$, each once.
Similarly, a semi-difference~$(p(v,s) - p(v,s'))/2$
takes twelve values
$\pm e(i) \pm e(j), \pm e(i) \pm e(k), \pm e(j) \pm e(k)$, each once.
Note that the semi-sum does not depend on the ordering,
but the semi-difference changes its sign.
\end{lemma}

It will be used in what follows,
together with the following result.

\begin{proposition} \label{Lemma-exactly-twice}
Let~$v \in N$ be a point in the lattice~$N$.
If there are rays~$p,q$ and~$p',q'$, such that the pairs~$\{p,q\}$ and~$\{p',q'\}$
are distinct, and~$p + q = 2v = p' + q'$, then for any pair of rays
$p'',q''$ such that~$2v = p'' + q''$ the pair~$\{p'',q''\}$ coincides
either with~$\{p,q\}$ or with~$\{p',q'\}$. In other words,
if some lattice point can be represented as semi-sum of two rays in
at least two different ways, then it can be represented so in exactly two ways.
\end{proposition}

\begin{proof}
  Suppose that for some vertices \(u,u',v,v' \in V(\graph)\) and signs \(x,y,z,w\)
  there is a relation
  \begin{equation}
    p(u,x) + p(u',y) = p(v,z) + p(v',w)
  \end{equation}
  For each vertex consider the parity of the number of times it appears in the relation.
  The relation implies that for every edge \(i\) the parities of the adjacent vertices are equal.
  So the assigned parities of vertices are constant in every connected component.
  Since we have only four terms in a relation,
  the number of odd vertices is at most four,
  so in every connected component with at least~$5$ vertices all vertices are even.

  Thus either \(u=u'=v=v'\),
  but then the six semi-sums are all distinct,
  or there are two vertices and each appears twice in the relation.
  Hence, either each vertex appears only on one side of the equation,
  or both vertices appear on both sides.

  If each vertex appears on its own side
  then we have an equation \(\pm e(i) = \pm e(j)\)
  where \(i\) is an edge adjacent to \(u\)
  and \(j\) is an edge adjacent to \(v\),
  thus \(i=j\) and the vertices \(u,v\) are adjacent to each other.

  Finally, the equation
  \begin{equation}
    \label{ssd}
    p(v,x) + p(v',w) = p(v,y) + p(v',z)
  \end{equation}
  is equivalent to \(p(v,x) - p(v,y) = p(v',z) - p(v',w)\)
  (where the differences are zero if and only if the summands on both sides are identical).
  For every vertex \(v\) with incident edges \(i,j,k\) the \(12\) semi-differences are equal to
  \begin{equation}
    \pm e(i)\pm e(j),\pm e(i)\pm e(k),\pm e(j)\pm e(k).
  \end{equation}
  So if for two distinct vertices \(v,v' \in V(\graph)\) and some sign choices \(x,y,z,w\)
  the relation \eqref{ssd} holds,
  then the vertices \(v\) and \(v'\) would be connected by at least two edges,
  contradicting our assumption about trivalent graphs.
\end{proof}

\section{Reconstruction of graphs from polytopes}\label{section:reconstruction}

The reconstruction is provided by the following procedure,
which a priori is defined for any polytope~$P$,
but will not work on polytopes not associated with any graph.
Recall that we refer to the vertices of a polytope as \emph{rays},
to avoid ambiguity with the vertices of a graph.

\begin{enumerate}
  \item Consider the rays~$r_1,\dots, r_n$ of the polytope~$P$ and find the center of mass~$c_0=\frac{1}{n}(r_1+\dots+r_n)$ of its rays.
    Let~$N_\R$ be a tangent space to the polytope at the interior point~$c_0$,
    and~$N\subset N_\R$ be its abelian subgroup generated by rays.
    If the polytope is associated with some graph,
    then~$N$ has to be a discrete lattice of full rank.
  \item
    Consider the multi-set~$\mathcal{S}$ of semi-sums of pairs of distinct rays
    which are lattice points,
    i.e.,~$\frac{1}{2}(p+q)$ for all~$p,q\in P$,
    provided that the semi-sum is in~$N$.
    Mark every semi-sum which appears at least twice.
  \item\label{enumerate:marking}
    For every marked semi-sum in~$\mathcal{S}$,
    consider it as a basis vector~$e(i)$ in the real vector space spanned by
    the semi-sums which appear at least twice,
    subject to the condition that
    if a semi-sum and its negative appear
    you assign them as~$e(i)$ and~$-e(i)$.
    The choice of sign will be irrelevant.
  \item Choose a basis among the~$e(i)$,
    by picking a sign for each pair of a vector and its negative.
  \item Let~$p(i)$ be the coordinate of the ray~$p\in P$.
\end{enumerate}
Then define the graph~$\graph'=(V',E')$, where
\begin{description}
  \item[vertices]~$V'$ is defined as the rays of the polytope
    modulo the equivalence relation which says that
    two rays are equivalent if in the basis given by the~$e(i)$
    they differ only up to signs;
  \item[edges]~$E'$ is given by the basis indexed by the~$e(i)$,
    and two vertices~$p$ and~$q$ are connected by an edge~$e(i)$
    if~$p(i)\neq0$ and~$q(i)\neq 0$.
\end{description}

Let us illustrate the procedure in the following toy example of genus~3.
Although the graph contains only~$4$ vertices,
the procedure still works.
Moreover, it is the unique example of a trivalent graph
without leaves, loops or multiple edges,
that has less than~$6$ vertices.
\begin{example}
  Consider the trivalent graph
  
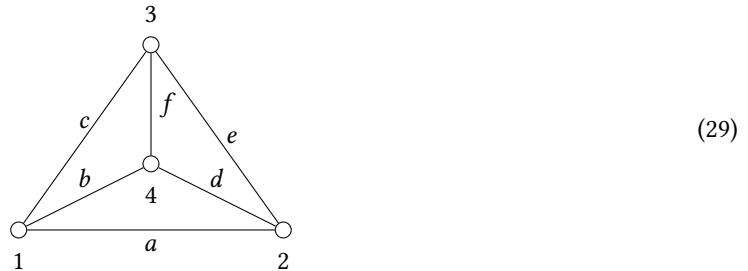
\begin{equation}
    \begin{tikzpicture}[scale=1.75,baseline=(current bounding box.center)]
      \node[vertex] (A) at (0,0) {};
      \node[vertex] (B) at (2,0) {};
      \node[vertex] (C) at (1,1.4) {};
      \node[vertex] (D) at (1,0.5) {};
      \draw (A) +(0,-0.1) node [below] {$1$};
      \draw (B) +(0,-0.1) node [below] {$2$};
      \draw (C) +(0,0.1) node  [above] {$3$};
      \draw (D) +(0,-0.1) node [below] {$4$};
      \draw (A) edge node [below] {$a$} (B);
      \draw (A) edge node [above] {$c$} (C);
      \draw (A) edge node [above] {$b$} (D);
      \draw (B) edge node [above] {$d$} (D);
      \draw (B) edge node [right] {$e$} (C);
      \draw (C) edge node [right] {$f$} (D);
    \end{tikzpicture}
  \end{equation}
  of genus~3, which thus has~4 vertices and~6~edges.
  For each vertex in the graph we consider the~4 possible sign choices,
  leading to the following~16 rays in the polytope:
  \begin{equation}
    \left\{
      \begin{gathered}
        ( 1, 1, 1,0,0,0) \\
        ( 1,-1,-1,0,0,0) \\
        (-1, 1,-1,0,0,0) \\
        (-1,-1, 1,0,0,0)
      \end{gathered}
    \right.
    ,\quad
    \left\{
      \begin{gathered}
        ( 1,0,0, 1, 1,0) \\
        ( 1,0,0,-1,-1,0) \\
        (-1,0,0, 1,-1,0) \\
        (-1,0,0,-1, 1,0)
      \end{gathered}
    \right.
    ,\quad
    \left\{
      \begin{gathered}
        (0,0, 1,0, 1, 1) \\
        (0,0, 1,0,-1,-1) \\
        (0,0,-1,0, 1,-1) \\
        (0,0,-1,0,-1, 1)
      \end{gathered}
    \right.
    ,\quad
    \left\{
      \begin{gathered}
        (0, 1,0, 1,0, 1) \\
        (0, 1,0,-1,0,-1) \\
        (0,-1,0, 1,0,-1) \\
        (0,-1,0,-1,0, 1)
      \end{gathered}
    \right..
  \end{equation}
  Considering the semi-sums of pairs of distinct rays,
  such that the semi-sum is a lattice point in the standard lattice,
  we obtain the following list,
  where we have already indicated a choice of basis.
  \begin{equation}
    \left\{
      \begin{aligned}
        e_1 &:=(1, 0, 0, 0, 0, 0), \\
        e_2 &:=(0, 1, 0, 0, 0, 0), \\
        e_3 &:=(0, 0, 1, 0, 0, 0), \\
        -e_3&:=(0, 0, -1, 0, 0, 0), \\
        -e_2&:=(0, -1, 0, 0, 0, 0), \\
        -e_1&:=(-1, 0, 0, 0, 0, 0)
      \end{aligned}
    \right.
    ,\quad
    \left\{
      \begin{aligned}
        e_1 &:=(1, 0, 0, 0, 0, 0), \\
        e_4 &:=(0, 0, 0, 1, 0, 0), \\
        e_5 &:=(0, 0, 0, 0, 1, 0), \\
        -e_5&:=(0, 0, 0, 0, -1, 0), \\
        -e_4&:=(0, 0, 0, -1, 0, 0), \\
        -e_1&:=(-1, 0, 0, 0, 0, 0)
      \end{aligned}
    \right.
    ,\quad
    \left\{
      \begin{aligned}
        e_3 &:=(0, 0, 1, 0, 0, 0), \\
        e_5 &:=(0, 0, 0, 0, 1, 0), \\
        e_6 &:=(0, 0, 0, 0, 0, 1), \\
        -e_6&:=(0, 0, 0, 0, 0, -1), \\
        -e_5&:=(0, 0, 0, 0, -1, 0), \\
        -e_3&:=(0, 0, -1, 0, 0, 0)
      \end{aligned}
    \right.
    ,\quad
    \left\{
      \begin{aligned}
        e_2 &:=(0, 1, 0, 0, 0, 0), \\
        e_4 &:=(0, 0, 0, 1, 0, 0), \\
        e_6 &:=(0, 0, 0, 0, 0, 1), \\
        -e_6&:=(0, 0, 0, 0, 0, -1), \\
        -e_4&:=(0, 0, 0, -1, 0, 0), \\
        -e_2&:=(0, -1, 0, 0, 0, 0)
      \end{aligned}
    \right.
   \end{equation}
   It is now clear that we can reconstruct the graph using the procedure outlined above,
   as the coordinates of the points in our chosen basis encode the adjacency relation.
\end{example}

\paragraph{The proof of correctness}
Let us now explain how the procedure
reconstructs the graph~$\graph$ from the polytope~$P(\graph)$.

If a vertex \(v\in V(\graph)\) is incident to edges \(i,j,k\)
then the \(6\) semi-sums of two vectors from a quadruple
are equal to \(\pm e(i),\pm e(j),\pm e(k)\).
So if \(i\) is an (internal non-loop) edge connecting vertices \(v\) and \(v'\)
there are at least two distinct ways to write \(e(i)\) as a semi-sum of two vectors:
\begin{equation}
  2 e(i) = p(v,x) + p(v,y) = p(v',z') + p(v',w'),
\end{equation}
and similarly
\begin{equation}
  -2 e(i) = p(v,z) + p(v,w) = p(v',x') + p(v',y').
\end{equation}

Now \cref{Lemma-exactly-twice} shows that no other vector has two representations as a semi-sum,
and finishes the proof of correctness of the algorithm.

\begin{proof}[Proof of \cref{theorem:main}]
The reconstruction procedure uniquely recovers the graph~$\graph$
from its polytope~$P = P(\graph)$,
thus proving \cref{theorem:main}.
\end{proof}
\section{Graph potentials}\label{section:gp}
In this section, we recall a construction
of a class of Laurent polynomials,
so-called \emph{graph potentials},
introduced in \cite{gp-tqft,gp-sympl,gp-decomp}.
We slightly simplify and expand the notation of loc.~cit.,
and collect some of their basic properties.

Let~$\graph$ be a trivalent graph of genus~$g$ and let~$V$ be its set of vertices.
Consider the lattice~$\widetilde{N}_{\graph}=\Z^E\cong\Z^{3g-3}$ as in
\cref{section:polytopes} and \cite[\S2.1]{gp-tqft}.
The graph potential~$\widetilde{W}_{\graph}$
associated to~$\graph$ is defined by
\begin{equation}
\label{equation:important-graph-potential}
\widetilde{W}_{\graph}
:=
\sum_{v\in V}\sum_{\substack{s_i,s_j,s_k\in\{\pm1\} \\ s_i s_j s_k = 1}}x_i^{s_i}x_j^{s_j}x_k^{s_k}
= \sum_{v,s} z^{p(v,s)}
\end{equation}
where~$x_i, x_j, x_k$ are variables attached to the edges~$e_i,e_j,e_k$
incident to the vertex~$v \in V$,
and $s = (s_i,s_j,s_k)$.
Vectors $p(v,s) = s_i e_i + s_j e_j + s_k e_k \in \Z^{E}$
are defined by \eqref{equation:rays},
and the sum is over all vertices $v$ and all admissible signs $s$;
note that admissibility coincides with \eqref{equation:sign-choice-no-coloring}.
Equivalently, $\phi_\graph(\theta)$ and $M_\graph(u)$ are defined by
\begin{equation}
\label{equation:important-graph-potential-phi-pvs}
 \phi_\graph(\theta) = \sum_{v,s} \exp(\mathrm{i}\langle p(v,s),\theta\rangle),
\end{equation}
\begin{equation}
\label{equation:important-graph-potential-M-pvs}
 M_\graph(u) = \sum_{v,s} \exp(\langle p(v,s),u\rangle).
\end{equation}

\begin{remark}
Recall that the vectors $p(v,s)$ generate a proper sublattice $N_\graph$ in $\widetilde{N}_\graph = \Z^E$.
So graph potentials $\widetilde{W}_\graph$ are associated with non-irreducible random walks.
In contrast, graph potentials $W_\graph$ are defined by restricting to $N$,
which corresponds to taking the quotient of the torus with coordinates $\theta$.
As we discussed in the introduction such restriction affects neither the shape of Newton polytopes, nor period sequences,
so we mostly ignore the difference between the irreducible potential $W_\graph$
and its non-irreducible counterpart $\widetilde{W}_\graph$
for the purposes of this paper.
In other papers,
$W_\graph$ is a proposed mirror dual to the $\mathrm{SU}(2)$-character variety,
whereas $\widetilde{W}_\graph$ is proposed mirror dual to the $\mathrm{PSU}(2) = \mathrm{SO}(3,\mathbb{R})$-character variety,
considered as an orbifold.
\end{remark}

The first observation is that the qCG polytope of \cref{section:polytopes}
is precisely the Newton polytope of the graph potential.

\begin{proposition}
\label{proposition:newton-polytope-qcg}
The Newton polytope of Laurent polynomials $\widetilde{W}_\graph$ and $W_\graph$
is the quantum Clebsch--Gordan polytope $P(\graph) \subset N_\graph \otimes \R$
defined in \cref{section:polytopes}. In particular,
the Newton polytope is $(3g-3)$-dimensional,
and the origin lies strictly inside of it.
\end{proposition}
\begin{proof}
This fact follows immediately from the definitions.
The support of $W_\graph$ is the set of all $p(v,s)$,
and $P(\graph)$ by definition is the convex envelope of all $p(v,s)$.
The properties of qCG polytopes are established in \cref{section:polytopes}.
\end{proof}

Combined with the reconstruction theorem of \cref{section:reconstruction},
this turns the Newton polytope into a complete invariant of the graph.

\begin{corollary}
\label{cor:distinct-shape}
For any pair of non-isomorphic graphs $\graph$ and $\graph'$
the shapes of the Newton polytopes of their graph potentials
$W_\graph$ and $W_{\graph'}$ are distinct.
\end{corollary}
\begin{proof}
By \cref{theorem:main} (proved in \cref{section:reconstruction})
the graphs $\graph$ and $\graph'$ can be reconstructed
from their Newton polytopes $\NP(W_\graph)$ and $\NP(W_{\graph'})$.
\end{proof}

Beyond the combinatorial information carried by their Newton polytopes,
graph potentials enjoy a basic analytic property:
the origin is a critical point of the (logarithmic) moment generating function.
This will be used in \cref{section:random-walks} to interpret~$\widetilde{W}_\graph$
as the characteristic function of a mean-zero random walk.

\begin{proposition}
\label{prop:gp-is-meanzero}
For any graph $\graph$ the point $u=0 \in \R^E$ (resp. $z=1$)
is a critical point of $M_\graph$
(resp. $\widetilde{W}_\graph$ and $W_\graph$).
\end{proposition}
\begin{proof}
For any finite collection of vectors $p_1,\ldots,p_n \in \R^E$,
the derivative of $\sum_{i=1}^n \exp{\langle p_i,u\rangle}$ at $u=0$
equals $\sum_{i=1}^n p_i$.
Since $M_\graph(u) = \sum_{v\in V}\sum_s \exp{\langle p(v,s),u\rangle}$,
where the inner sum is over admissible sign tuples~$s$,
we have
\begin{equation}
  \nabla M_\graph(0) = \sum_{v\in V}\sum_s p(v,s) = \sum_{v\in V} 0 = 0,
\end{equation}
where we used the equality $\sum_s p(v,s) = 0$ for each vertex~$v$,
established in \cref{lemma:center}.
\end{proof}

\paragraph{Change of coordinates}\label{section:changeofcoordinates}
We now turn to part~\cref{enumerate:principal-equal-invariants} of \cref{theorem:principal},
and show that the graph potentials defined above
produce the same periods, density-of-states functions and other integral invariants~$K_p(t)$,~$\zeta_p(s)$.

We say that two trivalent graphs~$\graph$ and~$\graph'$ are related by an elementary transformation at an edge
if they differ by a local surgery at a neighbourhood of an edge
as shown in \cref{figure:elementary-transformation-without-colors}.
Namely, if half edges~$a, b, c, d$ (not all distinct) are adjacent to the edge~$x$
in the two vertices~$v_1,v_2$ in pairs~$\{a,b\}$ and~$\{c,d\}$ in the graph~$\graph$,
then they meet the edge in pairs~$\{a,c\}$ and~$\{b,d\}$ in the graph~$\graph'$.
Another elementary transformation will be given by considering instead the pair~$\{a,d\}$ and~$\{b,c\}$.
Geometrically, they might look different, but this difference plays no extra role for our purpose.

\begin{figure}[t!]
  \centering

  \begin{tikzpicture}[scale=1]
  \node[vertex] (first)  at (0,0) {};
  \node[vertex] (second) at (1,0) {};
  \node[draw] (A) at (-1,1)  {};
  \node[draw] (B) at (-1,-1) {};
  \node[draw] (C) at (2,1)   {};
  \node[draw] (D) at (2,-1)  {};

  \draw (first)  +(0,-0.1) node [below] {$v_1$};
  \draw (second) +(0,-0.1) node [below] {$v_2$};

  \draw (A)     edge node [above] {$a$} (first);
  \draw (B)     edge node [below] {$b$} (first);
  \draw (first) edge node [above] {$x$} (second);
  \draw (C)     edge node [above] {$c$} (second);
  \draw (D)     edge node [below] {$d$} (second);

  \node at (3,0) {$\mapsto$};

  \node[vertex] (first)  at (5,0.5) {};
  \node[vertex] (second) at (5,-0.5) {};
  \node[draw] (A) at (4,1.5)  {};
  \node[draw] (B) at (4,-1.5) {};
  \node[draw] (C) at (6,1.5)  {};
  \node[draw] (D) at (6,-1.5) {};

  \draw (first)  node [left] {$v_1$};
  \draw (second) node [left] {$v_2$};

  \draw (A)     edge node [above] {$a$} (first);
  \draw (B)     edge node [below] {$b$} (second);
  \draw (first) edge node [right] {$y$} (second);
  \draw (C)     edge node [above] {$c$} (first);
  \draw (D)     edge node [below] {$d$} (second);
  \end{tikzpicture}
  \caption{Elementary transformation without colors}
  \label{figure:elementary-transformation-without-colors}
\end{figure}
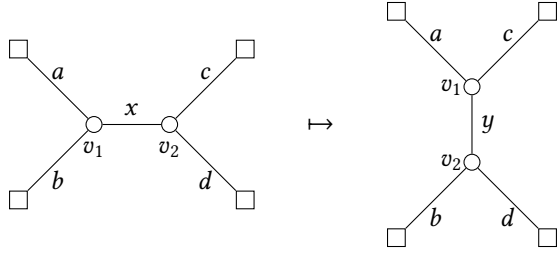

These transformations are also known as Hatcher--Thurston moves after \cite{MR0579573},
where it is proved that any two colored trivalent graphs can be related by a sequence of elementary transformations of edges. With this set-up it is shown in Section 2 of \cite{gp-tqft}, that if~$\graph$ and~$\graph'$ are related by an elementary transformation of edges as shown in \cref{figure:elementary-transformation-without-colors}, then the following rational transformation:
\begin{equation}
  x\cdot (ab+cd)=y\cdot (ac+bd)
\end{equation}
relates the graph potentials~$W_{\graph}$ and~$W_{\graph'}$.
We refer the reader to \cite[Theorem 2.12]{gp-tqft} for a precise statement.

\begin{proof}[Part~\cref{enumerate:principal-equal-invariants} of \cref{theorem:principal}]
  For~$K_p(t)$ and~$\zeta_p(s)$
  the result follows from a change of variables formula
  for smooth forms on $\R^E$.

  If~$v=F(u)$ is a diffeomorphism of~$\R^d$, then the absolute value of the Jacobian~$|\det \mathrm{D}F|$ is one.
  This is because in logarithmic coordinates on $\R^d$
  ($X = \log x$, $Y = \log y$, etc) the transformation formulas take the form
  \begin{equation}
    X = Y + f(A,B,C,D),
  \end{equation}
  where~$X, A,B,C,D,Y$ are variables on~$\R^d$, and
  \begin{equation}
    f = \log{(\exp(A+C)+\exp(B+D))}-\log{(\exp(A+B)+\exp(C+D))}
  \end{equation}
  is a smooth function on $\R^d$.
  Hence taking derivatives on both sides with respect to~$X$ gives
  \begin{equation}
    dX= dY + df(A,B,C,D).
  \end{equation}
  Thus we get that~$\frac{\partial Y}{\partial X}=1$.
  Since the transformation in the other variables is the identity, we get that
  $dA \wedge dB \wedge dC \wedge dD \wedge dX = dA \wedge dB \wedge dC \wedge dD \wedge dY$,
  hence the Jacobian is~$1$.

  Then by the change of variables formula, we get the following equalities:
  \begin{equation}
    \int_{u\in \R^d}\exp(-tM_{\graph}(u))\dd u=\int_{v\in \R^d}\exp(-tM_{\graph'}(v))\dd v
  \end{equation}
  and
  \begin{equation}
    \int_{u\in\R^d} M_{\graph}(u)^{-s} \dd u= \int_{v\in\R^d} M_{\graph'}(v)^{-s} \dd v,
  \end{equation}
  where~$M_{\graph}(F(u))=M_{\graph'}(v)$.

  For the period sequence~$\Pi=(a_m)_{m=0,1,\ldots}$ the same change of variables~$F$
  extends to complex arguments, but the induced map between the complex tori is only
  birational rather than a diffeomorphism. A change-of-variables formula at this level
  is established in \cite[Lemma~2.17 and Corollary~2.18]{gp-tqft}, applied to the graph
  potentials~$W_\graph$ and~$W_{\graph'}$ of two trivalent graphs related by an elementary
  transformation, and yields~$a_m(W_\graph)=a_m(W_{\graph'})$ for all~$m\geq 0$.
\end{proof}
\def\W{\overline{W}}

\section{Random walks on lattices}
\label{section:random-walks}

In this section we translate the setting to that of
finite range lattice random walks,
and explain how graph potentials of \cref{section:gp}
yield interesting examples of random walks,
whose ``shapes'' one cannot ``hear''
due to combinatorial reconstruction
of \cref{section:reconstruction}
and behavior of integrals
from \cref{section:gp}.
This translation was already considered in \cite{galkin-split-notes},
as a way to go back and forth between
random walks
and objects originating from mirror symmetry. 
We refer the reader to \cite[\S1.1]{MR2677157} for the definition of random walks.

We first explicitly discuss how a Laurent polynomial can be thought of as a random walk.

\begin{definition}Consider a Laurent polynomial $W$ 
  \begin{equation}
  \label{equation:random-walk-function}
  W(z_1,\dots, z_d) = \sum_{v\in \Z^d} p_v z^{v},
\end{equation}
in the ring of Laurent polynomials~$\R[N]$
associated to the lattice~$N \cong \Z^d$
in the~$d$\dash dimensional real vector space~$N_\R\cong\R^d$.
   satisfying the following assumptions:
  \begin{itemize}
    \item~$\sum_{v\in \Z^d} p_v=1$ and~$p_v\geq 0$ for all~$v\in \Z^d$;
    \item the origin~$0\in N$ is an interior point of~$P$.
  \end{itemize}
  The {\em associated random walk on the lattice~$N$} is defined in the usual way:
  the number $p_v$ is the probability to make a step in direction $v$.
\end{definition}

These are the normalised conditions on a Laurent polynomial
as considered in \cite{1404.7388,galkin-split-notes}.

Recall from the introduction the notion of Newton polytope associated to a Laurent polynomial
or random walk.
We make the following observation.

\begin{remark}
The vertices of the Newton polytope form the generating set for a random walk,
provided there are no other lattice points on the polytope other than the vertices.
This condition will hold for our irreducible walks, but it is not essential.
\end{remark}

Finally,
recall from the introduction that~$a_m$ is the return probability
\begin{equation}
  a_m:=P(S_m=0\mid S_0=0),
\end{equation}
and equals the constant term~$[W^m]_0$ of the~$m$th power of~$W$;
see \cite[Corollary~2.2.3]{MR2677157} and \cite[\S2.3]{gp-tqft}.

\paragraph{The conifold point and random walks}
The condition that a random walk associated to a Laurent polynomial
has mean zero and is \emph{non-degenerate}
(its steps do not lie in a proper affine subspace,
equivalently the covariance matrix is non-degenerate)
can be translated to a property of the Laurent polynomial~$W$
which was already considered in \cite{1404.7388,galkin-split-notes}.
We will now explain how this interpretation goes in the notation of op.~cit.,
which will set up a dictionary between random walks and Laurent polynomials
as they appear in mirror symmetry.

Let~$u_1,\ldots,u_d$ be coordinates on~$\C^d$.
The exponential map~$\C^d\rightarrow \left(\C^\times\right)^d$
is a topological covering map
and identifies~$\R^d \subset \C^d$ with~$\R^d_{+} \subset \left(\C^\times\right)^d$.
Let~$z_i=\exp(u_i)$ and we identify~$\R^d$ with~$\R^d_{+}$ via the exponential map.
Observe that the partial derivative~$\frac{\partial}{\partial u_i}$
coincides with~$z_i\frac{\partial }{\partial z_i}$ under this identification.

With this setup,
we can make the following elementary observation.
\begin{proposition}
  \label{proposition:balanced}
  Let~$W$ be a Laurent polynomial with associated random walk as above.
  The condition that it has mean zero,
  resp.~is non-degenerate,
  corresponds to
  \begin{itemize}
    \item the point~$c=(1,1,\ldots,1)$ (in exponential notation, i.e., the origin of~$N$)
    is the (unique) point of~$\R_+^d$
    for which all the logarithmic derivatives~$\partial W/\partial u_i$ vanish;
    \item the critical point~$c$ is a Morse point,
    i.e., its Hessian matrix~$(\frac{\partial^2W}{\partial u_i \partial u_j})_{i,j}$ is non-degenerate.
  \end{itemize}
\end{proposition}
The point~$c$ is thus the \emph{conifold point} studied in \cite{1404.7388,galkin-split-notes,MR4384381,2405.16979}.

\begin{proof}
  The first is a direct translation of the fact that the gradient of~$W$ evaluated at~$c$
  gives~$\sum_{v\in N}p_vv$, and thus describes the mean of the random walk.
  The second is another direct translation,
  by realising that the Hessian matrix of the Laurent polynomial
  corresponds to the covariance matrix of the associated random walk,
  which is invertible by \cite[Proposition~1.1.1]{MR2677157}.
  Uniqueness of a real positive critical point $u_0$
  comes for free, it is a corollary
  of convexity of $M(u)$ implied by positive-definiteness of covariance matrix,
  see \cite{1404.7388,galkin-split-notes}.
\end{proof}

\paragraph{The shape of a random walk}
Let~$N$ and~$N'$ be two lattices
and consider two random walks~$W$ and~$W'$ in~$N$ and~$N'$
satisfying the conditions of \cref{proposition:balanced}.
\begin{definition}
Let~$W$ and~$W'$ be finite range lattice random walks.
We say that they:
\begin{enumerate}
  \item have the \emph{same shape}
    if there exists a bijective linear map~$\psi\colon N_\R \rightarrow N'_\R$
    which maps the Newton polytope of~$W$ to the Newton polytope of~$W'$;
  \item have \emph{equal probabilities of return to the origin} (or \emph{same periods})
    if~$a_m(W)=a_m(W')$ for all~$m\in \bN$.
\end{enumerate}
\end{definition}

We ask the following question,
where throughout we assume that our random walks are as in \cref{proposition:balanced}.
\begin{question}
\label{question:fundamental}
Suppose two $d$-dimensional finite range lattice random walks~$W$ and~$W'$ have equal
probabilities of return to the origin.
Is it true that they are of the same shape?
\end{question}

Note that the existence of such isospectral but non-isomorphic structures
is a recurring theme in the modern spectral analysis of networks \cite{MR3237552}.

More generally,
if the answer to \cref{question:fundamental} is negative,
can we find a collection of random walks~$\{W_1,\dots, W_k\}$
which all have the same periods,
but pairwise they are of different shapes?
We will reinterpret our results from~\cref{section:reconstruction}
to answer \cref{question:fundamental} negatively,
producing arbitrarily large collections of such random walks.

\paragraph{Construction}
We give a construction of such mean-zero random walks
based on the graph potentials $\widetilde{W}_{\graph}$ and $W_\graph$
which we defined in \cref{section:gp} after \cite{gp-tqft}.
Consider the normalisation
\begin{equation}
  \label{equation:walk-construction}
  \W_{\graph} := \frac{1}{8(g-1)}\widetilde{W}_{\graph},
\end{equation}
so that we can consider~$\W_{\graph}$ as the defining data for a random walk:
it goes to one of $8g-8$ directions $p(v,s)$ with equal probability $\frac{1}{8g-8}$.
By \cref{prop:gp-is-meanzero} the random walk associated to~$\W_\graph$ is mean-zero,
and by construction its support generates~$N_\graph$ as a $\Z$-module
(see~\cite[\S2.1]{gp-tqft}), so the walk is irreducible, and in particular non-degenerate.
By \cite[Lemma~1.4]{galkin-split-notes},
the convexity of a positive Laurent polynomial
implies that it has a unique critical point on~$\R_+^d$, attained at the global minimum.
Together with \cref{proposition:balanced} this gives the following.
\begin{proposition}
  \label{proposition:conifold-property}
For each graph~$\graph$,
  the Laurent polynomial~$\W_{\graph}$ defines a random walk in the lattice~$N_{\graph}$
  which is mean-zero and irreducible,
  and it has the conifold property.
\end{proposition}
In particular, this proves part~\ref{enumerate:principal-zero-mean} of \cref{theorem:principal}.

\paragraph{Comparing periods and Newton polytopes}
Now we investigate the question of comparing periods
and Newton polytopes of graph potentials,
associated to different graphs (of a fixed genus).
We can reinterpret \cite[Corollary~2.9]{gp-tqft}  as follows.
\begin{proposition}
\label{prop:same-sound}
Let~$\graph$ and~$\graph'$ be two graphs of the same genus.
Then for all~$m$, the probabilities of return~$a_m$ to the origin after~$m$ steps of the associated random walks are the same.
\end{proposition}
\begin{proof}
  As recalled above, see also \cite[\S2.3]{gp-tqft},
  the probability~$a_m$ of return to the origin after~$m$ steps of the random walk associated
  to a Laurent polynomial~$W$ equals the constant term~$[W^m]_0$,
  which is the~$m$th term of the classical period sequence~$\Pi_W$.
  Equality of the return probabilities of~$\W_\graph$ and~$\W_{\graph'}$ is therefore equivalent
  to equality of the period sequences~$\Pi_{W_\graph}$ and~$\Pi_{W_{\graph'}}$,
  which is the period case of part~\cref{enumerate:principal-equal-invariants} of \cref{theorem:principal},
  proved in \cref{section:gp}.
\end{proof}

\begin{remark}
Observe that the number of trivalent graphs of genus~$g$ grows very quickly \cite{oeis-A005967}.
Each one of them produces random walks~$\W_{\graph}$.
Thus we are producing many random walks with the same probabilities.
\end{remark}

Finally, we can reinterpret~\cref{theorem:main} to answer
the amplification of~\cref{question:fundamental}.
This will be further generalized to the colored case in \cref{theorem:fullness} of \cref{section:colorings}.
\begin{theorem}
\label{theorem:cannot-hear}
The random walks~$\W_{\graph_i}$
associated with a collection of pairwise-distinct graphs~$\graph_i$
have pairwise-distinct shapes, but pairwise-equal probabilities of return to the origin.
\end{theorem}
\begin{proof}
Combine \cref{prop:same-sound}
with \cref{cor:distinct-shape}.
\end{proof}

\paragraph{Symmetric random walks}
A random walk~$W=\sum_{v\in N}p_vz^v$ is \emph{symmetric}
if~$p_v=p_{-v}$ for every~$v\in N$.
Equivalently,~$W$ is real-valued on the unit torus,
so multiplication by~$W$ on~$\mathrm{L}^2(\bT^d)$ is self-adjoint.
Every symmetric random walk has mean zero.
The graph walks used in \cref{theorem:cannot-hear} are not symmetric;
in the colored setting, the only symmetric example is the odd coloring
of the genus~$2$ Theta graph.
The following special case of \cref{question:fundamental} therefore remains open.
\begin{question}
  \label{question:symmetric-random-walks}
  Can one hear the shape of a symmetric lattice random walk?
\end{question}

This question asks only whether the return probabilities determine the Newton polytope.
The stronger problem of recovering both the polytope and the probabilities of its steps
already has a negative answer in dimension~$2$.
Indeed, consider the symmetric random walks
\begin{equation}
  \label{equation:symmetric-random-walk-example}
  \begin{aligned}
    W(x,y)
    &=\frac{1}{24}\left((x+5+6x^{-1})y+(x^{-1}+5+6x)y^{-1}\right),\\
    W'(x,y)
    &=\frac{1}{24}\left((3x+7+2x^{-1})y+(3x^{-1}+7+2x)y^{-1}\right).
  \end{aligned}
\end{equation}
They have the same support and hence the same shape,
namely the square~$[-1,1]^2$,
but their probabilities on the upper edge, read from left to right,
are respectively~$\frac{1}{24}(6,5,1)$ and~$\frac{1}{24}(2,7,3)$;
the probabilities on the lower edge are forced by symmetry.
These triples are neither equal nor reversals,
so no linear symmetry of the common support identifies the distributions.

To see that their return probabilities agree, set~$A=x+5+6x^{-1}$
and~$A'=3x+7+2x^{-1}$.
The factorisations~$A=x^{-1}(x+2)(x+3)$
and $A'=x^{-1}(x+2)(3x+1)$
give~$A(x)A(x^{-1})=A'(x)A'(x^{-1})$,
because~$(x+3)(x^{-1}+3)=(3x+1)(3x^{-1}+1)$.
The odd return probabilities vanish,
while for every~$k\geq0$ the even ones are
\begin{equation}
  \label{equation:symmetric-random-walk-periods}
  [W^{2k}]_0
  =\frac{1}{24^{2k}}\binom{2k}{k}
  \left[(A(x)A(x^{-1}))^k\right]_0
  =[(W')^{2k}]_0.
\end{equation}
Thus the return probabilities do not distinguish the two probability distributions,
but this example does not answer \cref{question:symmetric-random-walks}:
the two Newton polytopes are equal.
\section{Interpretation in Fermi varieties}
\label{section:fermi}
Another reinterpretation of our results appears in the context of algebraic Fermi varieties, for another particular choice of a discrete Laplacian.

Consider the lattice~$\Z^d$ and a sublattice~$\smash{\Gamma=\bigoplus_{i=1}^d \Z a_ie_i}$
where the~$e_i$ denote the standard basis vectors.
Consider the~$a_1\cdots a_d$-dimensional complex Hilbert space~$\mathrm{L}^2(\Z^d/\Gamma)$
of complex-valued~$\Gamma$-periodic functions on~$\Z^d$
with the inner product of two functions~$\phi$ and~$\psi$ defined as
\begin{equation}
  \langle \phi, \psi \rangle:=\sum_{x\in \Z^d/\Gamma} \phi(x)\overline{\psi}(x).
\end{equation}
The potential function~$q(x)$ in the setup of
Gieseker--Kn\"orrer--Trubowitz \cite{MR1099878}
lies in the Hilbert space~$\mathrm{L}^2(\Z^d/\Gamma)$
and they define the following discrete Laplacian
\begin{equation}
  \Delta=\sum_{j=1}^d S_j + S_j^{-1},
\end{equation}
where the~$S_j$'s are the shift operators acting on functions by the formula
\begin{equation}
  S_j \psi(x_1,\dots,x_d):=\psi(x_1,\dots, x_{j-1}, x_j+1, x_{j+1},\dots, x_d).
\end{equation}
Gieseker--Kn\"orrer--Trubowitz introduce the following variety which they call the \emph{Bloch variety}
\begin{equation}
  B(q):=
  \{ (t_1,\dots, t_{d}, \lambda)\in (\C^\times)^d\times \C
  \mid
  \exists \psi \neq 0 \text{ a solution to the discrete model \eqref{eqn:discrete}} \}
\end{equation}
for the system
\begin{equation}
  \label{eqn:discrete}
  \begin{cases}
    (-\Delta +q-\lambda)\psi=0,
    & \text{where~$\lambda\in \C$} \\
    S^{a_j}\psi=t_j\psi,
    & \text{where~$t_j\in \C^\times$ and~$1\leq j\leq d$.}
  \end{cases}
\end{equation}
This discrete model is an analog of the model from solid-state physics called the independent electron model.
The role of~$q(x)$ is played by a field potential,
$\Gamma$ by ions
and the~$\psi$ are wave functions for a gas of electrons moving in~$\R^d/\Gamma$
under the influence of~$q(x)$.
\eqref{eqn:discrete} is the Schr\"odinger equation
with periodic boundary conditions. Physicists consider the 
(real) Fermi variety~$F_{\lambda}$ inside~$(\mathrm{S}^1)^d\times \R$
as the fiber over~$\lambda \in \R$ under the second projection
from~$(\mathrm{S}^1)^d\times \R$.
When~$d=2$, they are referred to as \emph{Fermi curves}.

The following is one of the main theorems in \cite{MR1184395} and should be interpreted as a version of the Torelli-type reconstruction theorems,
see also \cite[Theorem~4.1]{MR1099878}.
\begin{theorem}[Gieseker--Kn\"orrer--Trubowitz]
  If~$d=2$ and~$a_1$,~$a_2$ are distinct odd primes, the Bloch variety determines 
  the generic potential up to obvious symmetries.
  More precisely, there exists a Zariski-dense subset~$L \subset \mathrm{L}^2(\Z^2/\Gamma)$ such that,
  if~$B(q)=B(q')$ for~$q\in L$ and~$q' \in \mathrm{L}^2(\Z^2/\Gamma)$
  then there exists~$(x_0,y_0) \in \Z^2/\Gamma$
  for which
  \begin{equation}
    q'(x,y)=q(\pm x+x_0, \pm y + y_0).
  \end{equation}
\end{theorem}

\paragraph{Specialization to graph potentials}
The analogy also includes the \emph{shape} of the objects involved.
The definition of the Bloch variety and Fermi varieties
starts from a discretization of a Laplacian.
Likewise,
one can show that the graph potentials from \cite{gp-tqft}
(see also \cref{section:random-walks})
can be interpreted in terms of a discretization of the standard Laplacian in~$\R^{3g-3}$.
This also gives a proof of part~\cref{enumerate:principal-approximation} of \cref{theorem:principal}.

For this we
consider the normalised graph potential in \eqref{equation:walk-construction},
and consider its series expansion around the point~$(1,\ldots,1)$ in logarithmic coordinates over the real numbers,
i.e.,
after having applied~$\exp\colon N_\R\to\R_+^{3g-3}$ in the setup of \eqref{equation:important-graph-potential-M-pvs}.
The special properties of this point are discussed in \cref{proposition:balanced}.
The expansion~$\overline{W}_{\graph,c}(\exp(\hbar\partial))$ has constant term~$1$,
no linear term,
and the coefficient for~$\hbar^2$ or the quadratic term is precisely the standard Laplacian.

This can be checked
from the definition in \eqref{equation:important-graph-potential},
where the sum over the different sign choices allows us to cancel and rewrite accordingly.
Thus the translation of~$\overline{W}_{\graph,c}(\exp(\hbar\partial))$ to cancel the constant term
is truly a discretization of the (rescaled) standard Laplacian in~$\R^{3g-3}$.
We leave the details to the interested reader,
as they play no role in what follows.
We do include a proof of the appropriate part of \cref{theorem:principal}.

\begin{proof}[Part~\cref{enumerate:principal-fermi} of \cref{theorem:principal}]
  We first treat birationality.
  By \cref{section:changeofcoordinates}, the graph potentials of two trivalent graphs~$\graph$ and~$\graph'$
  related by an elementary transformation satisfy~$W_{\graph'}=W_{\graph}\circ F$
  for a birational change of coordinates~$F$ of the torus,
  see \cite[Theorem~2.12]{gp-tqft}.
  As~$F$ preserves the value of the potential, for every~$\lambda$ it restricts to a birational map
  \begin{equation}
    F_\lambda(W_{\graph'})=\{W_{\graph'}=\lambda\}\dashrightarrow\{W_{\graph}=\lambda\}=F_\lambda(W_{\graph}),
  \end{equation}
  and any two trivalent graphs of a fixed genus are connected by a sequence of such transformations.

  We now treat non-isomorphism.
  Fix~$\lambda$ for which~$W_{\graph}-\lambda$ is reduced; this excludes at most finitely many values.
  The Fermi variety~$F_\lambda(W_{\graph})=V(W_{\graph}-\lambda)$ is a hypersurface in the torus~$(\C^\times)^{3g-3}$,
  so it recovers the Laurent polynomial~$W_{\graph}-\lambda$ up to a unit,
  that is, up to a monomial factor and a nonzero scalar.
  Multiplying by a monomial only translates the Newton polytope,
  and since the origin lies strictly inside~$\NP(W_{\graph})$ by \cref{proposition:newton-polytope-qcg},
  the constant term~$-\lambda$ leaves it unchanged; thus
  \begin{equation}
    \NP(W_{\graph}-\lambda)=\NP(W_{\graph})=P(\graph),
  \end{equation}
  the quantum Clebsch--Gordan polytope, recovered up to a lattice translation.
  Hence an isomorphism~$F_\lambda(W_{\graph})\xrightarrow{\sim}F_\lambda(W_{\graph'})$
  induced by an automorphism of the torus would carry~$W_{\graph}-\lambda$
  to a unit multiple of~$W_{\graph'}-\lambda$,
  and would therefore induce a unimodular affine isomorphism~$P(\graph)\xrightarrow{\sim}P(\graph')$.
  By \cref{cor:distinct-shape}, itself a consequence of the reconstruction theorem in~\cref{theorem:main},
  this forces~$\graph\cong\graph'$.
  Therefore, for non-isomorphic graphs no monomial change of coordinates relates the two Fermi varieties,
  which is the asserted non-isomorphism as subvarieties of the torus.
\end{proof}
\section{Groupoids of colored graphs and their polytopes}
\label{section:colorings}
In this section we will upgrade the context to include colorings,
and prove \cref{theorem:colored-reconstruction}
as \cref{theorem:fullness},
generalizing \cref{theorem:main} from \cref{section:reconstruction}
(which was already used in \cref{section:random-walks} to obtain \cref{theorem:cannot-hear}).

\paragraph{Colored graphs and their isomorphisms}
A \emph{colored graph} is a pair \((\graph,c)\) consisting of
a graph \(\graph\)
and a coloring function \(c\colon V(\graph) \to \{\pm1\}\).
We will consider~$\{\pm1\}$ as a group under the multiplication.
Whenever needed we will reinterpret~$\{\pm1\}$ as~$\Z/2\Z$.
We define the following.
\begin{definition}
  \label{definition:colored-isomorphism}
  An \emph{isomorphism} of colored graphs \((\graph,c)\) and \((\graph',c')\)
  consists of a pair \((f,g)\),
  where
  \begin{itemize}
    \item~$f$ is a graph isomorphism \(f\colon \graph \isomto \graph'\),
    \item~$g$ is a function \(g\colon E(\graph') \to \{\pm1\}\)
      such that for any vertex \(v\in V(\graph)\)
      with edges \(i',j',k'\) incident to \(f(v)\in V(\graph')\) the relation
      \begin{equation}
        c'(fv) / c(v) = g(i') g(j') g(k')
      \end{equation}
      holds.
  \end{itemize}
\end{definition}
In particular if \(f\) is the identity
we can regard the colorings \(c,c'\) as \(0\)-chains of the graph
with values in~$\Z/2\Z$,
which we will consider as a ring.
Then the set of all~$g$'s such that~$(\operatorname{id}, g)$ is an isomorphism
is the set of~$1$-chains whose boundary is the ratio of \(c'\) and \(c\) in \(\{\pm1\}\),
so the set of automorphisms is a torsor under the first relative homology group
\(\mathrm{H}_1(\graph,\partial \graph;\Z/2\Z)\).

\paragraph{Quantum Clebsch--Gordan polytope of a colored graph}
Next we generalise the definition~$P(\graph)$ to incorporate the coloring function.
\begin{definition}
  \label{definition:polytope-colored}
  Define the \emph{colored quantum Clebsch--Gordan polytope} \(P(\graph,c)\)
  as the convex envelope in \(S(\graph)\)
  of the rays \(p(v,s)\) in \eqref{equation:rays} where we now ask
  \begin{equation}
    \label{equation:sign-choice}
    s(i)s(j)s(k) = c(v).
  \end{equation}
\end{definition}
Note that \(P(\graph) = P(\graph,c_0)\),
where \(c_0\) is the constant coloring taking \(c_0(v) = +1\) for all~$v\in V(\graph)$,
as \eqref{equation:sign-choice} reduces to \eqref{equation:sign-choice-no-coloring}.

\paragraph{Groupoid of colored graphs}
\label{groupoid-of-colored-graphs}
The composition of two colored isomorphisms \((f,g)\colon (\graph,c) \to (\graph',c')\)
and \((f',g')\colon (\graph',c') \to (\graph'',c'')\) is \((f'\circ f, g'')\),
where~$g''$ is given by
\begin{equation}
  g''(i'') := g'(i'') \cdot \prod_{f'(i') = i''} g(i').
\end{equation}
Generators for the isomorphism groupoid of colored graphs can be chosen to be
\begin{equation}
  (f,g)\colon (\graph,c) \to (\graph',c')
\end{equation}
with either \(c = c' \circ f, g = 1\)
or \(f\) being the identity and \(g\) being an assignment of~$\pm 1$ to a single edge.

\paragraph{Functor of quantum Clebsch--Gordan polytopes}
\label{functor-of-quantum-clebsch-gordan-polytopes}
Any isomorphism
\begin{equation}
  (f,g)\colon (\graph,c) \to (\graph',c')
\end{equation}
of colored graphs naturally induces an isomorphism \(P(f,g)\colon P(\graph,c) \to P(\graph',c')\)
of the associated colored quantum Clebsch--Gordan polytopes.
The generators of the first kind correspond to ``permutations''
and the generators of the second kind correspond to linear transformations
that send \(e(i)\) to \(-e(i)\) for the chosen edge~\(i\) and preserve all other basis vectors.

\paragraph{Colored combinatorial non-abelian Torelli theorem}
We can upgrade the reconstruction result of \cref{section:reconstruction} as follows,
which is the more precise version of \cref{theorem:colored-reconstruction}.
\begin{theorem}
  \label{theorem:fullness}
  The association \(P(-)\) is a \emph{full} functor
  from the isomorphism groupoid of colored trivalent graphs
  to the groupoid of convex polytopes and affine isomorphisms.
  That is, any affine linear isomorphism between
  colored quantum Clebsch--Gordan polytopes \(P(\graph,c)\) and \(P(\graph',c')\)
  is induced by a uniquely defined isomorphism of colored graphs.
\end{theorem}

\begin{proof}
  The polytopes~$P(\graph,c)$ and~$P(\graph',c')$ are considered
  as in \cref{lemma:center},
  thus affine linear isomorphisms do not take translations into account.
  The steps of the algorithm above can be applied to
  the colored set-up without modifications.
  Then for any of the \(2^{\dim S(\graph)}\) bases of \(S(\graph)\)
  with given coordinate axes \(\pm e(i)\)
  the coloring of the vertices of the graph is well-defined and uniquely determined
  by the parity of the number of minus signs in
  the coordinates of the respective rays in the given basis.
  A unique choice of signs corresponds to the original coloring,
  all other choices produce homologically equivalent colorings,
  provided that the total parity (even or odd) of vertices in every leafless connected component remains unchanged.

  It suffices to notice that the change of the sign of a single vector in the basis corresponds to
  simultaneous change of the color of the two vertices adjacent to this edge.
\end{proof}
\section{Algebro-geometric interpretation}
\label{section:algebraic-geometry}

\paragraph{Abelian and non-abelian Torelli}
The usual Torelli theorem recovers a curve~$C$ from the pair~$(\Jac C,\Theta)$
of its Jacobian variety polarised by the Theta divisor \cite{Torelli},
see \cite[\S VI.3]{MR0770932} for a modern discussion.
More generally,
the ``Torelli package'' also provides a description of the polarised automorphism group
in terms of automorphisms of~\(C\),
and the full package also
classifies all isomorphisms between the pairs~\((\Jac C,\Theta)\) and~\((\Jac C',\Theta')\).

A similar \emph{non-abelian} Torelli package exists
for moduli spaces of bundles with non-abelian structure groups,
such as moduli spaces~$\moduli_C(r,d)$ and~$\moduli_C(r,\mathcal{L})$
of semistable vector bundles of higher rank~$r>1$ and degree~$d$
(resp.~fixed determinant~$\mathcal{L}\in\Pic C$).

To state the non-abelian Torelli package,
observe that
\begin{itemize}
  \item
    an isomorphism of curves~\(f\colon C \isomto C'\)
    induces an isomorphism~\(\moduli_C(r,\mathcal{L}) \isomto\moduli_{C'}(r,f_*\mathcal{L})\)
    given by \(\mathcal{E}\mapsto f_*\mathcal{E}\),
  \item
    a line bundle~\(\mathcal{N} \in \Pic C\)
    induces an isomorphism~\(\moduli_C(r,\mathcal{L}) \isomto \moduli_C(r,\mathcal{L}\otimes \mathcal{N}^{\otimes r})\)
    given by \(\mathcal{E} \mapsto \mathcal{E} \otimes\mathcal{N}\).
\end{itemize}
These isomorphisms generate a groupoid,
with objects being pairs \((C,\mathcal{L})\)
of a smooth, connected projective curve \(C\) of genus~$g\geq 2$
and a line bundle \(\mathcal{L}\in\Pic C\)
and morphisms from \((C,\mathcal{L})\) to \((C',\mathcal{L}')\)
given by triples \((f,\mathcal{N}',\iota)\) of
\begin{itemize}
  \item an isomorphism of curves \(f\colon C\isomto C'\),
  \item a line bundle \(\mathcal{N}'\in\Pic C'\),
  \item an isomorphism of line bundles \(\iota\colon f_*\mathcal{L}\isomto \mathcal{L}' \otimes \mathcal{N}'^{\otimes r}\).
\end{itemize}
Hence for fixed~$r\geq 2$
the association \((C,\mathcal{L}) \mapsto\moduli_C(r,\mathcal{L})\)
is promoted to a functor~$\mathcal{M}$ from this groupoid
to the groupoid of algebraic varieties and isomorphisms.

Unlike Jacobians the spaces \(\moduli_C(r,\mathcal{L})\) have anticanonical polarizations,
and a non-abelian Torelli theorem for smooth curves was formulated
in rank~2 independently by Mumford--Newstead \cite[Corollary]{MR0234958}
and Tyurin~\cite[Theorem~1]{MR0260745},
and in arbitrary rank by Narasimhan--Ramanan \cite[Theorem~3]{MR384797}.
The Torelli package was completed by a description of the automorphism group
by Kouvidakis--Pantev~\cite[Theorems~B and~D]{MR1336336},
leading to the following statement.
Let us phrase it for~$r=2$,
the general statement is similar but requires more notation.

\begin{theorem}[Non-abelian Torelli]\label{thm:nat}
  The functor~$\mathcal{M}\colon(C,\mathcal L)\mapsto\moduli_C(2,\mathcal L)$ is fully faithful:
  every isomorphism~$\moduli_C(2,\mathcal L)\isomto\moduli_{C'}(2,\mathcal L')$
  is induced by a unique isomorphism of pairs~$(C,\mathcal L)\isomto(C',\mathcal L')$.
  In particular~$\moduli_C(2,\mathcal L)\cong\moduli_{C'}(2,\mathcal L')$ implies~$C\cong C'$.
  The automorphism group~$\Aut(\moduli_C(2,\mathcal L))$ is an extension
  of the automorphisms of the curve~$C$
  by the $2$-torsion subgroup~$\Jac(C)[2]$ of its Jacobian.
\end{theorem}

\paragraph{From geometry to combinatorics (and back)}
In the combinatorial versions of these theorems
trivalent graphs will correspond to curves,
colorings to line bundles,
polytopes to moduli spaces,
and the functor \(P\) from \cref{section:colorings}
is a version of the functor~$\mathcal{M}$.

The analogy between curves and trivalent graphs is standard
and can be explained as follows.
With a graph one can associate a \emph{graph curve}~$C(\graph)$
as studied by Bayer--Eisenbud~\cite{MR1097026},
the union of projective lines enumerated by the vertices of the graph
with the intersections corresponding to the edges of the graph.
The algebro-geometric property of very ampleness of the canonical bundle
of graph curves in op.~cit.~corresponds to three-connectedness
of the respective trivalent graphs, and forbids loops and multiple edges.

The leaves (if there are any) correspond to additional marked smooth points.
These graphs, and their associated graph curves,
index the finite set of zero-dimensional strata
in the usual stratification of
the Deligne--Mumford--Knudsen moduli space of stable marked curves,
with the graph~$\graph$ being the dual graph of the marked curve~$C(\graph)$.

A path between such a zero-dimensional stratum and the interior of the moduli space corresponds to
a degeneration of a smooth algebraic curve to the graph curve.
Topologically, the respective vanishing cycles give the Thurston cut system
of the corresponding Riemann surface into trinions (i.e.,~spheres with three holes, or pairs-of-pants)
encoded by the same graph.

Moduli spaces such as Jacobians and~$\moduli_C(r,\mathcal{L})$ vary together with the smooth curve \(C\),
and we can try to construct limits or degenerations as the curve tends to a graph curve~$C(\graph)$.
It is natural to ask whether there are Torelli-type theorems for these limiting varieties.

\paragraph{Combinatorial Torelli theorems}
For stable curves of genus~$g$ with rational components
Artamkin~\cite{MR2101281} constructed a compactification~$A(\graph)$
of the generalized Jacobian which is a toric variety whose polytope
is built from the combinatorial data of cycles in~$\graph$.
This combinatorial approach is closely related to the tropical Torelli theorem for Jacobians established by Mikhalkin and Zharkov \cite{MR2457739}.
Artamkin~\cite{MR2272775} furthermore proves that~$A(\graph)$ recovers the graph~$\graph$
for graphs which are~$k$-connected for~$k\geq 3$.
Similar Torelli-type theorems for the Albanese torus for~$3$-connected graph curves
have been proved by Caporaso--Viviani~\cite{MR2641941}; see also \cite{MR3046301} for the foundational Riemann--Roch theory for weighted graphs.

Manon~\cite{MR2928457} constructed a toric degeneration of
the moduli space of rank-2 bundles with fixed determinant
for any trivalent colored graph~$(\graph,c)$,
with the translation between op.~cit.
and our setup using colored graphs being explained in \cite[\S2.2]{gp-sympl}.
But in the spirit of results of Artamkin~\cite{MR2272775} and Caporaso--Viviani~\cite{MR2641941},
it is interesting to ask a Torelli-type question:
are these toric varieties uniquely determined by the maximally degenerate curves?

The polytopes~$P(\graph,c)$ defined in \cref{definition:polytope-colored}
are the polar duals of moment polytopes (also known as the fan polytopes) of the toric varieties constructed by Manon \cite{MR2928457,MR3861806},
see also \cite[\S2.2]{gp-sympl}. We write~$X(\graph,c)$ for the resulting toric variety.
Since a projective toric Fano variety with its anticanonical polarization
is uniquely determined by its moment polytope,
\cref{theorem:fullness} can be considered as
a non-abelian analog of a Torelli-type theorem as in \cref{thm:nat}
for the toric varieties constructed by Manon \cite{MR2928457}.
At the level of morphisms, a torus-equivariant isomorphism between two of these varieties,
considered up to the torus (equivalently, a toric isomorphism in the sense of \cite[\S3.3]{MR2810322}),
is by \cite[Theorem~3.3.4]{MR2810322} the same as a lattice isomorphism of their fans,
and hence as an affine isomorphism of the fan polytopes~$P(\graph,c)$.
This promotes the fullness of~$P$ in \cref{theorem:fullness} to the following statement,
where we work with connected graphs.

\begin{corollary}[Combinatorial non-abelian Torelli]\label{cor:imp}
  The functor~$(\graph,c)\mapsto X(\graph,c)$ sending a colored trivalent graph of genus~$g$
  to its toric degeneration is fully faithful, for toric varieties and toric isomorphisms
  (torus-equivariant isomorphisms modulo the torus):
  every isomorphism~$X(\graph,c)\isomto X(\graph',c')$
  is induced by a unique isomorphism~$\phi$ of colored graphs~$(\graph,c)\isomto(\graph',c')$.
  In particular~$X(\graph,c)\cong X(\graph',c')$ implies~$\graph\cong\graph'$,
  with~$c'=\phi_*c$ in~$\HH_0(\graph',\Z/2\Z)$.
  The group of toric automorphisms of~$X(\graph,c)$ modulo the torus is an extension
  of the automorphisms of the graph~$\graph$
  by~$\HH_1(\graph;\Z/2\Z)\cong(\Z/2)^g$, the $2$-torsion of the tropical Jacobian of~$\graph$
  and the combinatorial counterpart of~$\Jac(C)[2]$; cf.~\cref{section:colorings}.
\end{corollary}

The connectedness assumption in \cref{cor:imp} can be dropped and we will get an isomorphism of graphs and hence a bijection between its connected components and the equality of the coloring in the $0$-th homology will be with respect to that bijection.
This result complements other recent progress in the field, such as the reconstruction of graphs from their posets of quasistable divisors \cite{2309.04570}.
 
\renewcommand*{\bibfont}{\small}
\printbibliography

\emph{Pieter Belmans}, \url{p.belmans@uu.nl} \\
Mathematical Institute, Utrecht University, Budapestlaan 6, 3584CD Utrecht, Netherlands

\emph{Sergey Galkin}, \url{sergey@puc-rio.br} \\
PUC-Rio, Departamento de Matem\'atica, Rua Marqu\^es de S\~ao Vicente 225, G\'avea, Rio de Janeiro, Brasil

\emph{Swarnava Mukhopadhyay}, \url{swarnava@math.tifr.res.in} \\
School of Mathematics, Tata Institute of Fundamental Research, 1 Homi Bhabha Road, Navy Nagar, Colaba, Mumbai 400005, India
\end{document}